\documentclass[reqno,11pt]{amsart}

\usepackage{mathdots}
\usepackage{tikz}
\usepackage{tikz-cd}
\usepackage{amssymb}
\usepackage{amsgen}
\usepackage{amsmath}
\usepackage{amsthm}
\usepackage{mathrsfs}
\usepackage{cite}
\usepackage{amsfonts}

\newcommand{\End}{\mathop{\mathrm{End}}}

\newcommand{\Res}{\mathop{\mathrm{Res}}\nolimits}
\newcommand{\Ind}{\mathop{\mathrm{Ind}}\nolimits}

\newcommand{\J}{\mathrel{\mathscr J}} % J - relation
\newcommand{\R}{\mathrel{\mathscr R}} % R - relation
\newcommand{\eL}{\mathrel{\mathscr L}} % L - relation
\newcommand{\HH}{\mathrel{\mathscr H}}

\newcommand{\inv}{^{-1}}

\newcommand{\ov}[1]{\ensuremath{\overline {#1}}}
\newcommand{\til}[1]{\ensuremath{\widetilde {#1}}}

\newcommand{\wh}{\widehat}

\newcommand{\Hom}{\mathop{\mathrm{Hom}}\nolimits}

\usepackage{xcolor}

\newcommand{\GGM}{\equiv_{\mathsf{GGM}}}
\newtheorem{Thm}{Theorem}[section]
\newtheorem{Prop}[Thm]{Proposition}

\newtheorem{Lemma}[Thm]{Lemma}
{\theoremstyle{definition}
}
{\theoremstyle{remark}
}
\newtheorem{Cor}[Thm]{Corollary}
{\theoremstyle{remark}
}
{\theoremstyle{remark}
}

{\theoremstyle{remark}
}
{\theoremstyle{remark}
}
{\theoremstyle{remark}
}
{\theoremstyle{remark}
\newtheorem*{Claim*}{Claim}}

\numberwithin{equation}{section}

\title[Extensions of theorems of Gasch\"utz, \v{Z}mud$'$ and Rhodes]{Extensions of theorems of Gasch\"utz, \v{Z}mud$'$ and Rhodes on faithful representations}

\author{Benjamin Steinberg}
\address[B.~Steinberg]{%
    Department of Mathematics\\
    City College of New York\\
    Convent Avenue at 138th Street\\
    New York, New York 10031\\
    USA}
\email{bsteinberg@ccny.cuny.edu}

\thanks{The author was supported by a PSC CUNY grant and a Simons Foundation Collaboration Grant.}
\date{January 22, 2022}

\begin{document}

\begin{abstract}
Gasch\"utz (1954) proved that a finite group $G$ has a faithful irreducible complex representation if and only if its socle is generated by a single element as a normal subgroup; this result extends to arbitrary fields  of characteristic $p$ so long as $G$ has no nontrivial normal $p$-subgroup.  \v{Z}mud$'$ (1956) showed
that the minimum number of irreducible constituents in a faithful complex representation of $G$ coincides with the minimum number of generators of its socle as a normal subgroup; this result can also be extended to arbitrary fields of any characteristic $p$ such that $G$ has no nontrivial normal $p$-subgroup (i.e., over which $G$ admits a faithful completely reducible representation).

Rhodes (1969) characterized the finite semigroups admitting a faithful irreducible representation over an arbitrary field as generalized group mapping semigroups over a group admitting a faithful irreducible representation over the field in question.  Here, we provide a common generalization of the theorems of \v{Z}mud$'$ and Rhodes by determining the minimum number of irreducible constituents in a faithful completely reducible representation of a finite semigroup over an arbitrary field (provided that it has one).

Our key tool for the semigroup result is a relativized version of \v{Z}mud$'$'s theorem that determines, given a finite group $G$ and a normal subgroup $N\lhd G$, what is the minimum number of irreducible constituents in a completely reducible representation of $G$ whose restriction to $N$ is faithful.
\end{abstract}

\maketitle

\section{Introduction}

A very classical, and difficult, question in the representation theory of finite groups asks to determine the minimum degree of a faithful representation.    Nearly 10 years ago, the author and Mazorchuk~\cite{effective} undertook a detailed study of the analogous question for finite semigroups.  The situation is even more complicated for semigroups and our work ended up using a mixture of group character theory, finite lattices and some rudimentary algebraic geometry (in addition to semigroup theory).  However, the degree of a representation, i.e., the dimension of the corresponding module, is not the only natural measure of the complexity of a representation.  One could consider, instead, the length of the module, that is, the number of composition factors  (or in the completely reducible case, the number of irreducible constituents).  The analogous question is then to find the minimum length of a faithful representation.  If a group or semigroup has a split basic algebra over the ground field (that is, each irreducible representation is one-dimensional), then the minimum length and minimum degree of a faithful representation coincide.  In particular, the minimum length question for a $p$-group over a field of characteristic $p$ or for a nilpotent semigroup over any field is equivalent to the minimum degree question and seems to be quite hard.

However, if one avoids bad characteristics for a group, then there are nontrivial results determining the minimum length of a faithful representation.  Note that having a length one faithful representation is equivalent to having a faithful irreducible representation.  The question as to which groups have a faithful irreducible complex representation goes all the way back to Burnside~\cite{Burnsidebook}.  A nice history of this problem can be found in Szechtman's paper~\cite{Szechtman}.  A rough summary is that the first characterization was found by Shoda~\cite{Shoda}, but it wasn't in some sense entirely group theoretic.  A somewhat more group theoretic characterization was found by Weisner~\cite{Weisner} (and later rediscovered by Kochend\"orffer~\cite{Kochen}).  Nakayama showed that the characterizations originally proved in the nonmodular setting also apply in the modular setting~\cite{Nakayamafaithful}.  The most elegant characterization was found by Gasch\"utz~\cite{Gaschutz} who showed that a finite group $G$ has a faithful complex irreducible representation if and only if its socle is generated by a single element as a normal subgroup.  Szechtman gives a nice explanation of how to directly deduce the Shoda, Weisner and Gasch\"utz characterizations from each other~\cite{Szechtman}.  These results work over arbitrary fields so long as the group $G$ has no nontrivial normal $p$-subgroup, where $p$ is the characteristic of the field.  Let me remark that Gasch\"utz's proof is based on counting the sums of the squares of the dimensions of the faithful irreducible representations of $G$ in terms of the number of elements that generate the socle of $G$ as a normal subgroup, based on an inclusion-exclusion (or M\"obius inversion) argument that exploits a duality on the lattice of normal subgroup of $G$ contained in the socle of $G$.  Szechtman's paper~\cite{Szechtman} gives  a module theoretic proof.

The question of determining whether a group $G$ has a faithful complex representation with $k$ irreducible constituents was first addressed by Tazawa~\cite{Tazawa} using the same language as Shoda.  Nakayama considered arbitrary fields in~\cite{Nakayamafaithful}, again using the language of Shoda and Tazawa.  \v{Z}mud$'$~\cite{Zmud1} proved the more elegant result, generalizing Gasch\"utz, that a finite group $G$ has a faithful complex representation with $k$ irreducible constituents if and only if the socle of $G$ can be generated by $k$ elements as a normal subgroup.  Like Gasch\"utz his proof is based on an inclusion-exclusion argument.  Later~\cite{zmud2}, \v{Z}mud$'$ extended the result to arbitrary fields whose characteristic does not divide the order of $G$.  This approach is based on studying the primitive idempotents in a certain commutative semisimple subalgebra of the center of the group algebra of $G^k$.  However, we shall see in this paper that the approach of Szechtman to Gasch\"utz's theorem can be extended to prove that \v{Z}mud$'$'s theorem holds so long as $G$ has a faithful completely reducible representation over the field in question, that is, it has no nontrivial normal $p$-subgroup where $p$ is the characteristic of the field.  This also follows from Nakayama's work~\cite{Nakayamafaithful}, which basically shows in the Shoda language that the modular situation behaves like the nonmodular situation as long as the group has a faithful completely reducible representation.

Rhodes characterized~\cite{Rhodeschar} the finite semigroups admitting a faithful irreducible representation over an arbitrary field $K$.  A semigroup $S$ is generalized group mapping if  it acts faithfully on both the left and the right of some (distinguished) ideal $I$ which contains no proper, nonzero ideal. For example, the monoid $M_n(F)$ of $n\times n$ matrices over a finite field  $F$ is generalized group mapping where the ideal $I$ consists of the matrices of rank at most $1$.  If $S$ is generalized group mapping with distinguished ideal $I$, then there is a unique (up to isomorphism) maximal subgroup $G$ of $S$ contained in the nonzero elements of $I$; we say that $S$ is generalized group mapping over $G$.  Rhodes proved that $S$ admits a faithful irreducible representation over $K$ if and only if it is generalized group mapping over a group admitting a faithful irreducible representation over $K$. For example, if $F$ is a field of $q$ elements, the maximal subgroup of the rank $1$ matrices is isomorphic to $F^\times$, which being cyclic has a faithful irreducible representation over any field whose characteristic does not divide $q-1$.  Hence $M_n(F)$ has a faithful irreducible representation over any such field.

Rhodes also described in~\cite{Rhodeschar} the finite semigroups admitting a faithful completely reducible representation over a field of characteristic $0$.  This was extended by Almeida, Margolis, Volkov and the author to arbitrary fields in~\cite{AMSV}.  In this paper, our main result is to give a common generalization of the theorems of \v{Z}mud$'$ and of Rhodes by determining the minimum number of irreducible constituents in a faithful completely reducible representation of a finite semigroup (provided that it has one).  We also show that if $S$ has a faithful completely reducible representation, then the minimum length of a faithful representation is achieved by a completely reducible representation.  Hence, this also determines the minimum length of a faithful representation of $S$ in this setting.

Perhaps surprisingly, the key hurdle in generalizing \v{Z}mud$'$'s theorem to semigroups was purely group theoretic in nature. It turns out that one needs a relative version of \v{Z}mud$'$'s theorem.  Namely, given a finite group $G$ and a normal subgroup $N$, we need to determine the minimum number of irreducible constituents in a completely reducible representation of $G$ that restricts to a faithful representation of $N$.  So long as $N$ has no nontrivial normal $p$-subgroup (where $p$ is the characteristic of the field), this turns out to be the minimum number of normal generators for the intersection of $N$ with the socle of $G$.  Taking $G=N$ yields \v{Z}mud$'$'s theorem~\cite{zmud2} under weaker hypotheses.  Our proof is based on Szechtman's module theoretic proof of Gasch\"utz's theorem~\cite{Szechtman}, as it is not at all clear how to make the approach of~\cite{zmud2} work when the characteristic divides the order of the group.

The paper is organized as follows.  Section~\ref{s:group} is entirely group theoretic in nature and requires no knowledge of semigroups.  We review the socle of a finite group and its basic properties.  Of particular importance is the subgroup $A(G)$ generated by all the abelian minimal normal subgroups of $G$, which is a $\mathbb ZG$-module. We then study the Pontryagin dual of $A(G)\cap N$ where $N\lhd G$.  The key point is to show that the minimal number of generators of $A(G)\cap N$ and its dual as $\mathbb ZG$-modules is the same.  Then we argue, following the general idea of Szechtman's paper~\cite{Szechtman}, that the dual of $A(G)\cap N$ is $k$-generated as a module if and only if $G$ admits a completely reducible representation with $k$ irreducible constituents that restricts to a faithful representation of $N$.  We use Clifford's theorem and Frobenius reciprocity here, as well as the theory of Schur indices to reduce from an arbitrary field to an algebraically closed field.  We believe that restricted to the case $N=G$, this is a more natural proof than that of \v{Z}mud$'$~\cite{zmud2}, in addition to giving a more general result.

Section~\ref{s:semigroup} generalizes \v{Z}mud$'$'s theorem to semigroups.  Here, we quickly review some structure theory of finite semigroups and the Clifford-Munn-Ponizovsky theory of irreducible representations of finite semigroups, explaining how they are determined by irreducible representations of maximal subgroups. Then we recall Rhodes's insight connecting generalized group mapping semigroups with irreducible representations.  Finally, we prove the main theorem by applying the relativized \v{Z}mud$'$ theorem to the normal subgroup of a maximal subgroup consisting of those elements indistinguishable from the identity by irreducible representations of the semigroup coming from other maximal subgroups (which we characterize intrinsically).

\section{Faithful completely reducible representations of groups}\label{s:group}
Our goal in this section is to determine the minimal number of irreducible constituents in a completely reducible representation of a group $G$ that restricts to a faithful representation of a fixed normal subgroup $N\lhd G$.  When $N=G$, this boils down to a slight extension of a theorem of \v{Z}mud$'$~\cite{Zmud1,zmud2}.

 Throughout, we shall assume familiarity with basic aspects of group representation theory like Maschke's theorem, Frobenius reciprocity and, most importantly, Clifford's theorem.  Also familiarity with Wedderburn-Artin theory and the Jacobson radical is presumed.  A classic reference is~\cite{curtis}.

\subsection{Socles of finite groups}
Let $G$ be a finite group.  In fact, all groups will be assumed finite throughout this paper, although we may include the hypothesis for emphasis.
 By a \emph{minimal normal subgroup} of $G$, we mean a minimal nontrivial normal subgroup.  The subgroup generated by all the abelian minimal normal subgroups of $G$ is denoted $A(G)$ and the subgroup generated by all the nonabelian minimal normal subgroups of $G$ is denoted $T(G)$.  Note that $A(G)$ and $T(G)$ are  normal (in fact, characteristic) subgroups of $G$, essentially by definition. The \emph{socle} $S(G)$ of $G$ is the subgroup generated by all minimal normal subgroups of $G$.  Note that $S(G)=A(G)T(G)$ is a normal (indeed, characteristic) subgroup of $G$.  Our notation $A(G)$, $T(G)$ and $S(G)$ follows Huppert~\cite{Huppertchar}.

We collect here a number of standard facts about minimal normal subgroups and the socle.   These results can be extracted from  of~\cite[Theorem~4.3A]{dixonbook} or~\cite[Section~3.3]{robinson}; see also~\cite[Lemma~42.9]{Huppertchar}.

A useful property of minimal normal subgroups is the following direct complement property.

\begin{Prop}\label{p:complement}
If $M\lhd G$ is a minimal normal subgroup and $N\lhd G$ is any normal subgroup, then either $M\leq N$ or $M\cap N=\{1\}$, and so we have an internal direct product decomposition $MN=M\times N$.
\end{Prop}

Minimal normal subgroups have a rigid structure, as the following well-known result indicates.

\begin{Prop}\label{p:minimalnormal.struct}
Let $M\lhd G$ be a minimal normal subgroup.  Let $T\lhd M$ be a minimal normal subgroup of $M$.
\begin{enumerate}
\item $T$ is a simple group.
\item $M$ is an internal direct product $T_1\times\cdots\times T_k$ of minimal normal subgroups $T_i$ of $M$ that are conjugate to $T$ in $G$.
\item If $M$ is nonabelian, then $T_1,\ldots, T_k$ are all the minimal normal subgroups of $M$.
\end{enumerate}
\end{Prop}

In the case $M$ is abelian, the previous proposition implies that $T$ is a cyclic group of prime order.

\begin{Cor}\label{c:abel.min.norm}
Every abelian minimal normal subgroup of $G$ is isomorphic to $(\mathbb Z/p\mathbb Z)^n$ for some prime $p$ and $n\geq 1$.
\end{Cor}

The following theorem describes the structure of the socle of a finite group.

\begin{Thm}\label{t:socle.struct}
Let $G$ be a finite group.
\begin{enumerate}
\item $T(G)$ is the internal direct product $T_1\times\cdots\times T_r$  of all the  nonabelian minimal normal subgroups $T_1,\ldots, T_r$  of $G$ (possibly $r=0$).
\item $Z(T(G))=Z(T_1)\times \cdots\times Z(T_r)=\{1\}$.
\item $S(G)$ is the internal direct product $A(G)\times T(G)$, that is, $A(G)\cap T(G)=\{1\}$.
\item $A(G)$ is an internal direct product $A_1\times\cdots\times A_s$ where $A_1,\ldots, A_s$ are abelian minimal normal subgroups; in particular, $A(G)$ is abelian.
\item $S(G)$ is an internal direct product $M_1\times\cdots\times M_t$  for some minimal normal subgroups $M_1,\ldots, M_t$ of $G$.
\end{enumerate}
\end{Thm}

If $A\lhd G$ is an abelian normal subgroup, then it is a $\mathbb ZG$-module where $G$ acts as automorphisms of $A$ via conjugation.
If $A$ is any finite $\mathbb ZG$-module and $p$ is a prime divisor of $|A|$, then since the $p$-Sylow subgroup $A_p$ of $A$ is characteristic, it is a $\mathbb ZG$-submodule of $A$.  Moreover, if $A_p$ is an elementary abelian $p$-group, that is, $pA=\{0\}$ (using additive notation) for all $a\in A_p$, then $A_p$ is naturally an $\mathbb F_pG$-module, where $\mathbb F_p$ is the $p$-element field,  and the $\mathbb ZG$-module structure factors through the $\mathbb F_pG$-module structure.

\begin{Prop}\label{p:AG.semisimple}
We have $A(G)=\bigoplus_{p\mid |A(G)|} A(G)_p$ as a $\mathbb ZG$-module.  Moreover, $A(G)_p$ is an elementary abelian $p$-group and is a semisimple $\mathbb F_pG$-module.
\end{Prop}
\begin{proof}
The direct sum decomposition comes from the structure theory of finite abelian groups.  If $M$ is an abelian minimal normal subgroup, then it is an elementary abelian $p$-group for some $p\mid |A(G)|$ by Corollary~\ref{c:abel.min.norm}.  It follows that $A(G)_p$ is generated by the minimal normal subgroups of $G$ that it contains and hence each nontrivial element of $A(G)_p$ has order $p$, again by Corollary~\ref{c:abel.min.norm}.  Thus $A(G)_p$ is an elementary abelian $p$-group and hence an $\mathbb F_pG$-module.  But a minimal normal subgroup of $G$ contained in $A(G)_p$ is the same thing as a simple $\mathbb F_pG$-submodule, and so the fact that $A(G)_p$ is generated by the minimal normal subgroups of $G$ that it contains implies that $A(G)_p$ is generated by simple $\mathbb F_pG$-modules and hence is a semisimple $\mathbb F_pG$-module.
\end{proof}

We now wish to analyze how the socle interacts with normal subgroups of $G$.  The following result is well known (cf., the proof of~\cite[Lemma~42.10]{Huppertchar}), but we include a proof for completeness and to make (3) explicit.

\begin{Prop}\label{p:interact.with.normal}
Let $N\lhd G$ be a normal subgroup.  Let $T_1,\ldots, T_r$ be the nonabelian minimal normal subgroups of $G$.
\begin{enumerate}
\item $S(G)\cap N = (A(G)\cap N)\times (T(G)\cap N)$.
\item $T(G)\cap N = \prod_{i\in X} T_i$ for some $X\subseteq \{1,\ldots,r\}$.
\item $A(G)_p\cap N = (A(G)\cap N)_p$ and $A(G)_p\cap N$ is a semisimple $\mathbb F_pG$-module for each $p\mid |A(G)\cap N|$.
\item $A(G)\cap N=\bigoplus_{p\mid |A(G)\cap N|} A(G)_p\cap N$ as a $\mathbb ZG$-module and this is the direct sum decomposition into Sylow subgroups.
\end{enumerate}
\end{Prop}
\begin{proof}
For the first and second items, let $X=\{i\mid T_i\leq N\}$.  Let $H$ be the subgroup generated by $A(G)\cap N$ and the $T_i$ with $i\in X$.  Then $H\lhd G$ and $H\leq S(G)\cap N$.  Suppose $g\in S(G)\cap N$ and write $g=at_1\cdots t_r$ with $a\in A(G)$ and $t_i\in T_i$, as $S(G)=A(G)\times T_1\times\cdots\times T_r$ by Theorem~\ref{t:socle.struct}.  Suppose that $t_i\neq 1$.  Then since $Z(T_i)=\{1\}$  by Theorem~\ref{t:socle.struct}, we can find $u\in T_i$ with $ut_iu\inv \neq t_i$  Then $1\neq ut_iu\inv t_i\inv =ugu\inv g\inv\in N$, and so $T_i\cap N\neq \{1\}$. Therefore, $T_i\leq N$ by Proposition~\ref{p:complement}, and so $i\in X$.  It follows that $t_1\cdots t_r\in N$ and so $a=g(t_1\cdots t_r)\inv\in A(G)\cap N$.  Therefore, $g=at_1\cdots t_r\in H$ and so $H=S(G)\cap N$.  The first two items now follow.

For the third item, obviously $A(G)_p\cap N = (A(G)\cap N)_p$.  Since submodules of semisimple modules are semisimple, the third item  follows from Proposition~\ref{p:AG.semisimple}.  The fourth item is immediate from the third.
\end{proof}

Let us agree that a module (respectively, normal subgroup) is \emph{$k$-generated} (respectively, \emph{$k$-generated} as a normal subgroup) if it can be generated by $k$ not necessarily distinct elements; that is, it can be generated by $k$ or fewer elements.  If $A$ is an abelian normal subgroup of our group $G$, then $A$ is $k$-generated as a $\mathbb ZG$-module if and only if it is $k$-generated as a normal subgroup of $G$.

\begin{Prop}\label{p:finite.zgmod.gen}
Let $A$ be a finite $\mathbb ZG$-module.  Then $A=\bigoplus_{p\mid |A|} A_p$ as a $\mathbb ZG$-module and, moreover, $A$ is $k$-generated if and only if each $A_p$ is $k$-generated.
\end{Prop}
\begin{proof}
The direct sum decomposition comes from the structure theory of finite abelian groups.  If $A$ is $k$-generated, then so are each of its quotients $A_p$.  Conversely, if each $A_p$ is $k$-generated, then we can find, for each prime divisor $p$ of $|A|$, elements $a_{1,p},\ldots, a_{k,p}$ generating $A_p$ as a $\mathbb ZG$-module.  Put $a_i = \sum_{p\mid |A|}a_{i,p}$.  Then $A$ is generated by $a_1,\ldots, a_k$ as a $\mathbb ZG$-module.  Indeed, by the Chinese remainder theorem, we can find $n_p\in \mathbb Z$ such that $n_p\equiv 1\bmod {o(a_{i,p})}$ (where $o(a)$ denotes the order of $a$) and $n_p\equiv 0\bmod q$ if $q$ is a prime divisor of $|A|$ different than $p$.  Then $n_pa_i =a_{i,p}$, and so $a_1,\ldots, a_k$ generate $A$ as a $\mathbb ZG$-module.
\end{proof}

Putting everything together, we arrive at the following proposition.

\begin{Prop}\label{p:kgen.soc}
Let $G$ be a finite group, $N\lhd G$ a normal subgroup and $k\geq 1$.  The following are equivalent.
\begin{enumerate}
\item $S(G)\cap N$ is $k$-generated as a normal subgroup.
\item $A(G)\cap N$ is $k$-generated as a normal subgroup of $G$.
\item $A(G)\cap N$ is $k$-generated as a $\mathbb ZG$-module.
\item $A(G)_p\cap N$ is $k$-generated as an $\mathbb F_pG$-module for each prime divisor $p$ of $|A(G)\cap N|$.
\end{enumerate}
\end{Prop}
\begin{proof}
Since $S(G)\cap N$ is the internal direct product $(A(G)\cap N)\times (T(G)\cap N)$ by Proposition~\ref{p:interact.with.normal}, if $S(G)\cap N$ is $k$-generated as a normal subgroup, then the same is true for $A(G)\cap N$ by considering the $(A(G)\cap N)$-components of the generators of $S(G)\cap N$.  So (1) implies (2).  Trivially (2) is equivalent to (3).  Since $A(G)_p\cap N$ is $k$-generated as a $\mathbb ZG$-module if and only if it is $k$-generated as an $\mathbb F_pG$-module, (3) is equivalent to (4) by Propositions~\ref{p:interact.with.normal} and~\ref{p:finite.zgmod.gen}.  It remains to show that (2) implies (1).

By Proposition~\ref{p:interact.with.normal}(1), there is nothing to prove if $T(G)\cap N=\{1\}$; so assume that this is not the case and
let $a_1,\ldots, a_k$ generate $A(G)\cap N$ as a normal subgroup.  Write $T(G)\cap N=T_1\times\cdots \times T_s$  (internal direct product) where $T_1,\ldots, T_s$ are nonabelian minimal normal subgroups of $G$ using Proposition~\ref{p:interact.with.normal}.  Choose $t_i\in T_i\setminus \{1\}$ for $i=1,\ldots, s$.  Put $a=a_1t_1\cdots t_s$.  We claim that $a,a_2,\cdots a_k$ generate $S(G)\cap N$ as a normal subgroup. Let $M$ be the normal subgroup generated by $a,a_2,\ldots, a_k$.  Since $Z(T_i)=\{1\}$, we can find $u_i\in T_i$ with $t_iu_it_i\inv u_i\inv\neq 1$.  Then $t_iu_it_i\inv u_i\inv=au_ia\inv u_i\inv\in T_i\setminus \{1\}\cap M$. It follows that $T_i\leq M$ by Proposition~\ref{p:complement}.  We conclude that $T(G)\cap N\leq M$ as $i$ was arbitrary.  But then $a_1=a(t_1\cdots t_s)\inv\in M$ and so $a_1,\ldots, a_k\in M$.  Thus $A(G)\cap N\leq M$ and so $S(G)\cap N=(A(G)\cap N)\times (T(G)\cap N)\leq M$ (using again Proposition~\ref{p:interact.with.normal}), as required.
\end{proof}

\subsection{The character module}
Let $F$ be an algebraically closed field.  If $A$ is a finite abelian group, the \emph{character group} (or Pontryagin dual) of $A$ over $F$ is the abelian group of characters $\Hom_{\mathbb Z}(A,F^\times)$ under pointwise multiplication.  There is a natural homomorphism $\eta\colon A\to \wh{(\wh{A})}$ given by $\eta(a)(f)=f(a)$.  The following proposition is a well-known consequence of the structure theorem for finite abelian groups.  We include a proof for completeness.

\begin{Prop}\label{p:Pont.duality}
Let $A$ be a finite abelian group and $F$ an algebraically closed field whose characteristic does not divide $|A|$.
\begin{enumerate}
  \item $\eta\colon A\to \wh{(\wh A)}$ is an isomorphism.
  \item $A\cong \widehat{A}$ as abelian groups.
  \item Let $B\leq \wh A$ be a subgroup.  Then $B=\wh{A}$ if and only if $\bigcap_{f\in B} \ker f$ is trivial.
\end{enumerate}
\end{Prop}
\begin{proof}
Let $p\geq 0$ be the characteristic of $F$.
If $C_n$ is a cyclic group of order $n$ and $p\nmid n$, then the $n^{th}$-roots of unity form a cyclic subgroup $\mu_n$ of $F^\times$ of order $n$, and so $\wh{C_n} = \Hom_{\mathbb Z}(C_n,\mu_n)\cong \Hom_{\mathbb Z}(C_n,C_n)\cong C_n$. This proves (2) for cyclic groups and it then follows in general by the structure theorem for finite abelian groups. In light of (2), we just need injectivity of $\eta$ to prove (1).  Let $0\neq a\in A$.  By the structure theorem for finite abelian groups, there is a surjective homomorphism $f\colon A\to C_n$ with $f(a)\neq 0$ and $C_n$ cyclic of order $n$.  But then there is an isomorphism $g\colon C_n\to \mu_n$ (since $n\mid |A|$), and so $gf\in \wh{A}$ with $\eta(a)(gf)= g(f(a))\neq 1$.  Finally, for (3), it follows from (1) that $\bigcap_{f\in \wh A}\ker f=\{0\}$.  If $B<\wh A$ is a proper subgroup, then $\wh A/B$ is a nontrivial finite abelian group whose order is not divisible by $p$ (by (2)).  Hence, by (2), we can find a nontrivial homomorphism $\alpha\colon \wh A/B\to F^\times$. Let $\pi\colon \wh A\to \wh A/B$ be the projection.  Then $\alpha\pi\in \wh{(\wh A)}$ is nontrivial and so there exists $0\neq a\in A$ with $\eta(a) = \alpha\pi$.  But  if $f\in B$, then $f(a)=\alpha\pi(f)=1$ and so $0\neq a\in \bigcap_{f\in B}\ker f$, as required.  This completes the proof.
\end{proof}

As a consequence, we obtain the following criterion for generation of $\wh A$.

\begin{Cor}\label{c:generate}
Let $A$ be a finite abelian group and $F$ an algebraically closed field whose characteristic does not divide $|A|$.  Then $f_1,\ldots, f_r$ generate $\wh A$ if and only if $\bigcap_{i=1}^r\ker f_i$ is trivial.
\end{Cor}
\begin{proof}
Let $B=\langle f_1,\ldots, f_r\rangle$.   We claim that $\bigcap_{f\in B}\ker f=\bigcap_{i=1}^r\ker f_i$.  Trivially, $\bigcap_{f\in B}\ker f\leq \bigcap_{i=1}^r\ker f_i$.  But if $a\in A$, then the set of $f\in \wh A$ with $f(a)=1$ is the kernel of $\eta(a)$, hence a subgroup, and so if $a\in \bigcap_{i=1}^r\ker f_i$, then $a\in \bigcap_{f\in B}\ker f$.  The result now follows from Proposition~\ref{p:Pont.duality}(3).
\end{proof}

If $G$ is a finite group and $A$ is a finite $\mathbb ZG$-module, then $\wh{A}$ becomes a $\mathbb ZG$-module via the action $(gf)(a)= f(g\inv a)$ for $f\in \wh A$, $a\in A$ and $g\in G$.  We call this the \emph{character module} of $A$.

Again, let $G$ be a finite group, $N\lhd G$ a normal subgroup and assume from now on that $F$ is an algebraically closed field whose characteristic does not divide $|A(G)\cap N|$.  As usual,  view $A(G)\cap N$ as a $\mathbb ZG$-module via the conjugation action.

\begin{Prop}\label{p:sylows.dual}
If $p\mid |A(G)\cap N|$ is prime, then $(\wh{A(G)\cap N})_p= \wh {A(G)_p\cap N}$.  Moreover, \[\wh{A(G)\cap N}= \bigoplus_{p\mid |A(G)\cap N|} \wh {A(G)_p\cap N}\] as $\mathbb ZG$-modules, and this is the decomposition into Sylow subgroups. Furthermore, $\wh{A(G)\cap N}_p$ is an $\mathbb F_pG$-module.
\end{Prop}
\begin{proof}
In light of the direct sum decomposition \[A(G)\cap N= \bigoplus_{p\mid |A(G)|} A(G)_p\cap N\] into Sylow subgroups from Proposition~\ref{p:interact.with.normal}, we have that $\wh{A(G)\cap N}=\bigoplus_{p\mid |A(G)|} \wh{A(G)_p\cap N}$ as $\mathbb ZG$-modules.  Since we are assuming that the characteristic of $F$ does not divide $|A(G)\cap N|$, we have that $\wh{A(G)\cap N}\cong A(G)\cap N$ and $\wh{A(G)_p\cap N}\cong A(G)_p\cap N$ by Proposition~\ref{p:Pont.duality}(2), and so we deduce that $(\wh{A(G)\cap N})_p= \wh {A(G)_p\cap N}$ and the direct sum decomposition is into Sylow subgroups.  Since $A(G)_p\cap N$ is an elementary abelian $p$-group, so is the isomorphic group $\wh{A(G)_p\cap N}$, and hence it is an $\mathbb F_pG$-module.
\end{proof}

Our next goal is to show that $\wh{A(G)_p\cap N}$ is isomorphic to the contragredient $\mathbb F_pG$-module of $A(G)_p\cap N$.  Recall that if $K$ is a field and $V$ is $KG$-module, then the \emph{contragredient} $KG$-module is $V^*=\Hom_K(V,K)$ with module structure given by $(gf)(v) = f(g\inv v)$ for $f\in V^*$, $g\in G$ and $v\in V$.  If $V$ is finite dimensional, then $V\cong (V^*)^*$ as a $KG$-module via the double dual map $\eta\colon V\to (V^*)^*$ given by $\eta(v)(f)=f(v)$.

\begin{Prop}\label{p:character.contra}
If $F$ is an algebraically closed field whose characteristic does not divide $|A(G)\cap N|$, then $\wh{A(G)_p\cap N}\cong (A(G)_p\cap N)^*$ as $\mathbb F_pG$-modules for any prime $p$ dividing $|A(G)\cap N|$.
\end{Prop}
\begin{proof}
Let $\mu_p$ the a group of $p^{th}$-roots of unity in $F^\times$.  Since the characteristic of $F$ is not $p$, this is a cyclic group of order $p$.  Then, as abelian groups, we have $\wh{A(G)_p\cap N} = \Hom_{\mathbb Z}(A(G)_p\cap N,F^\times) =\Hom_{\mathbb Z}(A(G)_p\cap N,\mu_p)\cong \Hom_{\mathbb Z}(A(G)_p\cap N,\mathbb F_p)=\Hom_{\mathbb F_p}(A(G)_p\cap N,\mathbb F_p)$ via identification of $\mu_p$ with the additive group of $\mathbb F_p$.  But this identification clearly preserves the $G$-action by functoriality of $\Hom$, and so $\wh{A(G)_p\cap N}\cong (A(G)_p\cap N)^*$ as $\mathbb F_pG$-modules.
\end{proof}

We now study the relationship between a $KG$-module and its contragredient.  In particular, we will show that if $V$ is a finite dimensional semisimple $KG$-module, then $V^*$ is also semisimple and $V$ is $k$-generated if and only if $V^*$ is $k$-generated.  The argument  essentially follows that of Szechtman~\cite{Szechtman} for the special case of cyclic modules.

\begin{Prop}\label{p:contra.simple}
Let $V$ be a simple $KG$-module.  Then $V^*$ is also simple.  Hence if $W\cong V_1\oplus\cdots \oplus V_r$ is a semisimple $KG$-module with the $V_i$ simple, then $W^*\cong V_1^*\oplus\cdots \oplus V_r^*$ is semisimple with the $V_i^*$ simple.
\end{Prop}
\begin{proof}
Let $U'$ be a $KG$-submodule of $V^*$.  Let $U=\bigcap_{f\in U'}\ker f$.  Then $U$ is a subspace of $V$.  In fact, it is a $KG$-submodule for if $u\in U$, $g\in G$ and $f\in U'$, then $f(gu) = (g\inv f)(u) =0$ as $g\inv f\in U'$.  Thus $U=\{0\}$ or $U=V$.  If $U=V$, then trivially $U'=\{0\}$. If $U=\{0\}$, then since the double dual map $\eta\colon V\to (V^*)^*$ is an isomorphism, we deduce that the only element of $(V^*)^*$ vanishing on $U'$ is the zero functional and so $U'=V^*$.  Thus $V^*$ is simple.  The remainder of the proposition is straightforward from the decomposition $(V_1\oplus\cdots \oplus V_r)^* = V_1^*\oplus\cdots \oplus V_r^*$ as $\Hom$ preserves finite direct sums.
\end{proof}

Next we connect the endomorphism rings of $V$ and $V^*$.  We use $R^{op}$ for the opposite ring of a ring $R$.   Notice that if $B$ is a $K$-algebra, then $B^{op}$ is a $K$-algebra of the same dimension since $K$ is in the center of $B$.

\begin{Prop}\label{p:end.contra}
Let $V$ be a finite dimensional $KG$-module.  Then we have $\End_{KG}(V^*)\cong \End_{KG}(V)^{op}$.
\end{Prop}
\begin{proof}
We claim that $T\colon V\to V$ is in $\End_{KG}(V)$ if and only if the dual map $T^*\colon V^*\to V^*$ belongs to $\End_{KG}(V^*)$.  Since $(T^*)^*=T$ under the natural identification of $(V^*)^*$ with $V$ via the double dual map $\eta$ and $(T_1T_2)^*=T_2^*T_1^*$, the result will follow.  Suppose first that $T\in \End_{KG}(V)$.  Then, for $g\in G$, $f\in V^*$ and $v\in V$, we have $T^*(gf)(v) = (gf)(Tv) = f(g\inv Tv) = f(T(g\inv v)) = T^*f(g\inv v) = g(T^*f)(v)$ and so $T^*(gf) = gT^*f$.  The converse follows from this case by placing $V^*$ (respectively, $T^*$) in the role of $V$ (respectively, $T$) and identifying $T$ with $(T^*)^*$ under the isomorphism $\eta\colon V\to (V^*)^*$.
\end{proof}

The following result holds for any finite dimensional algebra.

\begin{Prop}\label{p:kgen.ss.mod}
Let $A$ be a finite dimensional $K$-algebra.  Then a finite dimensional semisimple $A$-module $V$ is $k$-generated if and only if, for each simple $A$-module $S$, the multiplicity of $S$ as a composition factor of $V$ is at most $k\cdot \dim_D S = k\cdot \dim_K S/\dim_K D$ where $D=\End_A(S)$, which is a finite dimensional division algebra over $K$.
\end{Prop}
\begin{proof}
Let $J(A)$ be the Jacobson radical of $A$.  Then the $k$-generated semisimple $A$-modules are the quotients of $(A/J(A))^k$.  By Wedderburn's theorem, $A/J(A)\cong S_1^{n_1}\oplus \cdots \oplus S_r^{n_r}$ where $S_1,\ldots, S_r$ represent the isomorphism classes of simple $A$-modules and $n_i=\dim_{D_i} S_i$, where $D_i=\End_A(S_i)$ is a finite dimensional division algebra over $K$  by Schur's lemma.  Thus \[(A/J(A))^k\cong S_1^{kn_1}\oplus\cdots\oplus S_r^{kn_r}\] and the result follows.
\end{proof}

As a corollary, we obtain that a finitely generated semisimple $KG$-module $V$ and its contragedient have the same number of generators.

\begin{Cor}\label{c:kgen.dual}
Let $V$ be a finite dimensional semisimple $KG$-module.  Then $V$ is $k$-generated if and only if $V^*$ is $k$-generated.
\end{Cor}
\begin{proof}
Write $V\cong V_1^{m_1}\oplus\cdots\oplus V_r^{m_r}$ with the $V_i$ pairwise nonisomorphic simple $KG$-modules.  Then by Proposition~\ref{p:contra.simple}, we have $V^*\cong (V_1^*)^{m_1}\oplus\cdots\oplus (V_r^*)^{m_r}$ and $V_i^*$ is simple for $i=1,\ldots, r$, whence $V^*$ is semisimple.  Moreover, $V_i^*\ncong V_j^*$ if $i\neq j$ by consideration of double duals.  Suppose $V$ is $k$-generated.  Then by Proposition~\ref{p:kgen.ss.mod}, we must have $m_i\leq k\cdot \dim_K V_i/\dim_K D_i$ for $i=1,\ldots, r$ where $D_i = \End_{KG}(V_i)$.  But $\End_{KG}(V_i^*) \cong D_i^{op}$ by Proposition~\ref{p:end.contra}, and so we obtain
\[m_i\leq k\cdot \dim_K V_i/\dim_K D_i=k\cdot \dim_K V_i^*/\dim_K D_i^{op}\] and so $V^*$ is $k$-generated by another application of Proposition~\ref{p:kgen.ss.mod}.  The converse follows using that $V\cong (V^*)^*$ and $V^*$ is semisimple.
\end{proof}

Putting together the results of this subsection with the previous one, we arrive at the following theorem.

\begin{Thm}\label{T:desc.kgen.socle}
Let $G$ be a group, $N\lhd G$ a normal subgroup and $F$ an algebraically closed field whose characteristic does not divide $|A(G)\cap N|$.  Let $k\geq 1$.  Then the following are equivalent.
\begin{enumerate}
\item $S(G)\cap N$ is $k$-generated as a normal subgroup.
\item $A(G)\cap N$ is a $k$-generated $\mathbb ZG$-module.
\item $A(G)_p\cap N$ is a $k$-generated $\mathbb F_pG$-module for each prime divisor $p$ of $|A(G)|$.
\item $\wh{A(G)_p\cap N}$ is a $k$-generated $\mathbb F_pG$-module for each prime divisor $p$ of $|A(G)|$.
\item $\wh{A(G)\cap N}$ is a $k$-generated $\mathbb ZG$-module.
\item There exist $\chi_1,\ldots, \chi_k\in \wh {A(G)\cap N}$ such that $\bigcap_{i=1}^k\bigcap_{g\in G}\ker g\chi_i=\{1\}$.
\end{enumerate}
\end{Thm}
\begin{proof}
The equivalence of (1)--(3) is the content of Proposition~\ref{p:kgen.soc}.  The equivalence of (3) and (4) follows from Corollary~\ref{c:kgen.dual} as $A(G)_p\cap N$ is a semisimple $\mathbb F_pG$-module by Proposition~\ref{p:AG.semisimple} and $\wh{A(G)_p\cap N}\cong (A(G)_p\cap N)^*$ by Proposition~\ref{p:character.contra}.  The equivalence of (4) and (5) is an application of Propositions~\ref{p:finite.zgmod.gen} and~\ref{p:sylows.dual}.  By Corollary~\ref{c:generate}, the condition in (6) is equivalent to there being $k$-elements of $\wh{A(G)\cap N}$ whose translates under $G$ generate $\wh{A(G)\cap N}$ as an abelian group.  But this is the same as saying that $\wh{A(G)\cap N}$ is a $k$-generated $\mathbb ZG$-module.
\end{proof}

The equivalent conditions in~\cite{Tazawa} to \v{Z}mud$'$'s theorem~\cite{Zmud1,zmud2}, stem from the above equivalences.

\subsection{Faithful completely reducible representations}

 If $F$ is a field, a matrix representation $\rho\colon G\to GL_n(F)$ is \emph{faithful} if $\ker \rho$ is trivial.  This depends only on the equivalence class of $\rho$ and so we shall say than an $FG$-module is $G$-\emph{faithful} if it affords a faithful matrix representation of $G$.  Of course, if $H\leq G$, then every $FG$-module  $V$ restricts to an $FH$-module, and so we may say that $V$ is \emph{$H$-faithful} if it restricts to an $H$-faithful $FH$-module.  Recall that a representation afforded by a simple module is called \emph{irreducible} and a representation afforded by a semisimple module is called \emph{completely reducible}.

Let us start with a well-known lemma, cf.~\cite[Theorem~8.4]{LamBook}.

\begin{Lemma}\label{l:p.group.simple}
Let $G$ be a $p$-group and $F$ a field of characteristic $p$.  Then the only simple $FG$-module is the trivial module.
\end{Lemma}
\begin{proof}
We proceed by induction on $|G|$.  If $G$ is trivial, there is nothing to prove.   Else,  let $V$ be a simple $FG$-module and first observe that if $g\in G$, then we have $g^{p^n}=1$ for some $n\geq 1$ and hence the minimal polynomial of $g$ divides $x^{p^n}-1 = (x-1)^{p^n}$.  Thus $1$ is an eigenvalue of $g$.  In particular, if $g\in Z(G)$,  then the eigenspace of $g$ for eigenvalue $1$ is a nonzero $FG$-submodule (because $g$ commutes with all elements of $G$) and hence is $V$.  Therefore, $g$ acts trivially on $V$, and so $V$ is inflated from an $F[G/Z(G)]$-module.  Since $Z(G)\neq \{1\}$, by induction we conclude that $V$ is the trivial $F[G/Z(G)]$-module and hence is the trivial $FG$-module.
\end{proof}

Recall that if $G$ is a group and $p$ is a prime, then the \emph{$p$-core} of $G$, denoted $O_p(G)$, is the largest normal $p$-subgroup of $G$, that is, the product of all normal $p$-subgroups of $G$.  Of course, $O_p(G)$ is a characteristic subgroup of $G$ by its definition.

\begin{Prop}\label{p:faithful}
Let $G$ be a finite group, $N\lhd G$ a normal subgroup and $F$ a field of characteristic $p\geq 0$.  Then the following are equivalent.
\begin{enumerate}
\item $G$ admits a completely reducible representation over $F$ that restricts to a faithful representation of $N$.
\item $p$ does not divide $|A(G)\cap N|$.
\item $p=0$, or $N$ has no nontrivial normal $p$-subgroup.
\end{enumerate}
\end{Prop}
\begin{proof}
If $p=0$, then $FG$ is semisimple, by Maschke's theorem, and is clearly $N$-faithful.   The other conditions hold trivially when $p=0$, and so the result follows in this case.   Suppose from now on that $p>0$.

Assume first that  $G$ has an $N$-faithful semisimple module $V$.  If we have that $p\mid |A(G)\cap N|$, then $M=A(G)_p\cap N$ is a nontrivial normal $p$-subgroup of $G$.  By Clifford's theorem, the restriction of $V$ to an $FM$-module is semisimple, and it clearly is  $M$-faithful. But then $M$ must be trivial by Lemma~\ref{l:p.group.simple}. We conclude that $p\nmid |A(G)\cap N|$. So (1) implies (2).

 To prove that (2) implies (3), suppose that $N$ has a nontrivial normal $p$-subgroup. Then $O_p(N)$ is nontrivial.  Since $O_p(N)$ is characteristic in $N$, it is normal in $G$ and hence contains a minimal normal subgroup $M$ of $G$, which must be a $p$-group.  Since $Z(M)$ is a nontrivial characteristic subgroup of $M$, we must have that $M=Z(M)$, and so $M\leq A(G)\cap N$. It follows that $p\mid |A(G)\cap N|$.

 Finally, assume that (3) holds.    Then we claim that $FG/J(FG)$  is $N$-faithful where $J(FG)$ is the Jacobson radical of $FG$.   Let $\rho$ be a representation afforded by $FG/J(FG)$; it is completely reducible.   We claim that $\ker \rho$ is a $p$-group.  Indeed, if $g\in \ker \rho$, then $g-1\in J(FG)$ and so, since $J(FG)$ is nilpotent, we can find $m>0$ with $0=(g-1)^{p^m} =g^{p^m}-1$, i.e., $g^{p^m}=1$.  We conclude that  $\ker \rho$ is a normal $p$-subgroup and hence $N\cap \ker \rho$ is trivial by assumption, establishing (1).  This completes the proof.
\end{proof}

It is well known (cf.~\cite[Corollary~8.6]{LamBook}) that a group has a faithful completely reducible representation over a field of characteristic $p$ if and only if it has no nontrivial normal $p$-subgroup (this is the equivalence of (1) and (3) in the case $G=N$), and so the conditions of the above proposition are all equivalent to $N$ having a faithful completely reducible representation.

Our next goal is to characterize when $G$ has an $N$-faithful semisimple module with $k$ irreducible constituents (composition factors) over a field $F$.  By taking $G=N$, this contains as a special case Gasch\"{u}tz's theorem~\cite{Gaschutz} characterizing when a group has faithful irreducible representation and \v{Z}mud$'$'s theorem~\cite{Zmud1,zmud2}.  We first handle the case of an algebraically closed field and then show how to deduce the general case from that.  One simplifying advantage of algebraically closed fields is that if $G_1,G_2$ are finite groups and $F$ is an algebraically closed field, then the simple $F[G_1\times G_2]\cong FG_1\otimes_F FG_2$-modules are the (outer) tensor products $V_1\otimes_F V_2$ of a simple $FG_1$-module $V_1$ with a simple $FG_2$-module $V_2$.  The next two results are well known and can be found, for instance, in~\cite{Szechtman}.  We include proofs for completeness.

\begin{Prop}\label{p:snag.case}
Let $G$ be a simple nonabelian group and $F$ a field.  Then $G$ has a faithful irreducible representation over $F$.
\end{Prop}
\begin{proof}
Since $G$ is simple nonabelian, we must have that $A(G)=\{1\}$.  It follows then from Proposition~\ref{p:faithful} (with $N=G$) that $G$ admits a faithful completely reducible representation $\rho$.  As $G$ is nontrivial, $\rho$ must have a nontrivial irreducible constituent, and that constituent must be faithful by simplicity.
\end{proof}

We upgrade the previous result to direct products of simple nonabelian groups for the case of an algebraically closed field.

\begin{Prop}\label{p:prod.snag.case}
Let $H=T_1\times\cdots\times T_m$, where $T_1,\ldots, T_m$ are simple nonabelian groups, and let $F$ be an algebraically closed field.  Then $H$ has a faithful irreducible representation $\rho\colon H\to \mathrm{GL}_n(F)$ for some $n\geq 1$. Moreover, $Z(H)=\{1\}$ and hence $\rho(H\setminus \{1\})$ contains no scalar multiple of the identity matrix.
\end{Prop}
\begin{proof}
Since the $T_i$ are simple nonabelian groups, we have that $Z(H)=Z(T_1)\times \cdots\times Z(T_m)=\{1\}$. Obviously, $T_1,\ldots, T_m$ are minimal normal subgroups of $H$ and hence $H=T(H)=S(H)$.  We may then deduce from Theorem~\ref{t:socle.struct} that  $T_1,\ldots, T_m$ are all the minimal normal subgroups of $H$.

By Proposition~\ref{p:snag.case}, we can find a simple $FT_i$-module $V_i$ affording a faithful representation of $T_i$.  Then $V=V_1\otimes_F\cdots\otimes_F V_m$ is a simple $FH$-module using that $F$ is algebraically closed (this is a general fact about simple modules for tensor products of finite dimensional algebras over an algebraically closed field).  We claim that $V$ affords a faithful representation $\rho$.  Since $T_1,\ldots, T_m$ constitute all the minimal normal subgroups of $H$, it suffices to check that $\ker \rho\cap T_i=\{1\}$ for all $i$.  But if $1\neq g\in T_i$, then we can find $v_i\in V_i$ with $gv_i\neq v_i$ by $T_i$-faithfulness of  $V_i$.  If we choose any nonzero vector $v_j\in V_j$ for $j\neq i$, then we have that \[(1,\ldots, g,\ldots, 1)v_1\otimes \cdots v_m = v_1\otimes v_{i-1}\otimes gv_i\otimes v_{i+1}\otimes \cdots \otimes v_m\neq v_1\otimes \cdots \otimes v_m\] and so $\ker \rho\cap T_i=\{1\}$.

The final statement follows because $\rho$ is a faithful representation and $H$ has trivial center and hence no nontrivial element can map to a scalar matrix.
\end{proof}

Recall that if $M\lhd G$ is a normal subgroup, then $G$ acts on the set of isomorphism classes of finite dimensional $FM$-modules.  If $V$ is an $FM$-module and $g\in G$, then $gV$ denotes the $FM$-module with underlying vector space $V$ and module structure given by $m\cdot v = g\inv mgv$ for $m\in M$ and $v\in V$. Note that $V$ is simple if and only if $gV$ is simple as they have the same $M$-invariant subspaces.    Also observe that $V\cong W$ implies $gV\cong gW$ and that $h(gV)=(hg)V$.  If $\chi$ is the character of $V$, then $g\chi$ is the character of $gV$ where $(g\chi)(m) = \chi(g\inv mg)$.  In particular, if $F$ is algebraically closed and the characteristic of $F$ does not divide $|A(G)\cap N|$, then the action of $G$ on isomorphism classes of simple $F[A(G)\cap N]$-modules corresponds to its action on $\wh{A(G)\cap N}$ via taking a module to its character.  Returning to the general situation, if $V$ is an $FM$-module, then $I(V)=\{g\in G\mid gV\cong V\}$ is called the \emph{inertia subgroup} of $V$. Since $m_0V\cong V$ for $m_0\in M$ (via $v\mapsto m_0v$), we have that $M\lhd I(V)$.

If $H\leq G$ is a subgroup and $V$ is an $FH$-module, the \emph{induced module} is $\Ind_H^G V = FG\otimes_{FH} V$.   We write $\Res_H^G W$ for the restriction of an $FG$-module $W$ to $FH$. Recall that Frobenius reciprocity yields $\Hom_{FG}(\Ind_H^G V,W)\cong \Hom_{FH}(V,\Res_H^G W)$.

If $V$ is a simple $FG$-module and $M\lhd G$ is a normal subgroup, then by Clifford's theorem  if $W$ is any simple $FM$-submodule of $\Res_M^G V$, there exists $e\geq 1$ such that $\Res_M^G V\cong \bigoplus_{t\in T} (tW)^e$ where $T$ is a transversal to $G/I(V)$.

\begin{Prop}\label{p:build.faithful.rep.soc}
Let $G$ be a group and $N\lhd G$ a normal subgroup.
Suppose that $F$ is an algebraically closed field whose characteristic does not divide $|A(G)\cap N|$.  Let $\chi_1,\ldots,\chi_k\in \wh{A(G)\cap N}$ be such that $\bigcap_{i=1}^k\bigcap_{g\in G}g\chi_i=\{1\}$.  Then there exist simple $F[S(G)\cap N]$-modules $V_1,\ldots, V_k$ such that \[\bigoplus_{i=1}^k\bigoplus_{t\in T_i}tV_i\] is $(S(G)\cap N)$-faithful, where $T_i$ is a transversal to $G/I(V_i)$ for $i=1,\ldots, k$.
\end{Prop}
\begin{proof}
By Proposition~\ref{p:interact.with.normal}, $S(G)\cap N=(A(G)\cap N)\times (T(G)\cap N)$. Note that $T(G)\cap N$ is a direct product of simple nonabelian groups by Propositions~\ref{p:minimalnormal.struct} and~\ref{p:interact.with.normal}.  Therefore, by Proposition~\ref{p:prod.snag.case}, there is a $(T(G)\cap N)$-faithful simple $F[T(G)\cap N]$-module $W$, and, moreover, no nontrivial element of $T(G)\cap N$ acts by scalar multiplication under this representation.  Let $V_i$ be the outer product of the one-dimensional module associated to $\chi_i$ with $W$, which is a simple $F[S(G)\cap N]$-module.  Concretely, $V_i$ is $W$ as an $F$-vector space and the module structure is given by $(ag)\cdot v = \chi_i(a)gv$ for $a\in A(G)\cap N$ and $g\in T(G)\cap N$.  Choose a transversal $T_i$ to $G/I(V_i)$ for $i=1,\ldots, k$. Let $V= \bigoplus_{i=1}^k\bigoplus_{t\in T_i}tV_i$ and let $\rho\colon S(G)\cap N\to \mathrm{GL}(V)$ be the associated representation.  Note that if $a\in A(G)\cap N$, $g\in G$ and $w\in W$, then in $gV_i$ we have that $a\cdot w = \chi_i(g\inv ag)w$, and so $\Res^{S(G)\cap N}_{A(G)\cap N} gV_i$ is isomorphic to a direct sum of $\dim W$ copies of the one-dimensional module associated to $g\chi_i$.

If $ag\in S(G)\cap N$ with $a\in A(G)\cap N$, $g\in T(G)\cap N$ and $g\neq 1$, then there is $w\in W$ such that $gw\notin Fw$ (as no nontrivial element of $T(G)\cap N$ acts by scalar multiplication).  Then in  $V_1$, we  have $ag\cdot w= \chi_1(a)gw\neq w$ (else $gw\in Fw$).  Therefore, $ag\notin \ker \rho$. On the other hand, if $1\neq a\in A(G)\cap N$, then there is $g\in G$ and $i\in \{1,\ldots, k\}$ such that $(g\chi_i)(a)\neq 1$ by assumption.  But if $t\in T_i$ with $gI(V_i)=tI(V_i)$, then $gV_i\cong tV_i$, and so $g\chi_i=t\chi_i$ and $a$ acts on $tV_i$ via scalar multiplication by $(t\chi_i)(a)=(g\chi_i)(a)\neq 1$ by the observation at the end of the previous paragraph.  Thus $\rho(a)\neq 1$.  We conclude that $\rho$ is faithful.
\end{proof}

We can now state and prove the version of our first main theorem for algebraically closed fields.  In the case that $N=G$ and the characteristic of $F$ does not divide $|G|$, this result is due to \v{Z}mud$'$~\cite{Zmud1,zmud2} (see also Tazawa~\cite{Tazawa}), although the result for more general characteristics follows from the combination of the results of \v{Z}mud$'$~\cite{zmud2} and Nakayama~\cite{Nakayamafaithful}.  We will later use the theory of Schur indices to handle arbitrary fields.   Recall that $G$ has a completely reducible representation over $F$ that restricts to a faithful representation of $N$ if and only if the characteristic of $F$ does not divide $|A(G)\cap N|$ by Proposition~\ref{p:faithful}, and so we shall need to impose this condition.  Recall that the \emph{length} of a module $V$ is the number of composition factors in a composition series for $V$ (i.e., the length of the composition series).

\begin{Thm}\label{t:alg.closed.case}
Let $G$ be a finite group, $N\lhd G$ a normal subgroup and $F$ an algebraically closed field whose characteristic does not divide $|A(G)\cap N|$.  Let $k\geq 1$. Then the following are equivalent.
\begin{enumerate}
\item $S(G)\cap N$ is $k$-generated as a normal subgroup of $G$.
\item $A(G)\cap N$ is $k$-generated as a $\mathbb ZG$-module.
\item $A(G)_p\cap N$ is $k$-generated as an $\mathbb F_pG$-module for each prime divisor $p$ of $|A(G)\cap N|$.
\item $\wh{A(G)_p\cap N}$ is $k$-generated as an $\mathbb F_pG$-module for each prime divisor $p$ of $|A(G)\cap N|$.
\item $\wh{A(G)\cap N}$ is $k$-generated as a $\mathbb ZG$-module.
\item There exist $\chi_1,\ldots,\chi_k\in \wh{A(G)\cap N}$ such that $\bigcap_{i=1}^k\bigcap_{g\in G}\ker g\chi_i=\{1\}$.
\item $G$ has a completely reducible representation over $F$ with $k$ irreducible constituents that restricts to a faithful representation of $N$.
\item There is an $N$-faithful $FG$-module $V$ of length $k$.
\end{enumerate}
\end{Thm}
\begin{proof}
We already know that (1)--(6) are equivalent thanks to Theorem~\ref{T:desc.kgen.socle}. Obviously, (7) implies (8) since the corresponding semisimple module is a direct sum of $k$ simple modules and hence has $k$ composition factors.  Next we turn to (8) implies (6).  Suppose that $V$ is an $N$-faithful $FG$-module of length $k$.  Let $S_1,\ldots, S_k$ be the composition factors of $V$.   Then since $F[A(G)\cap N]$ is semisimple, by our hypothesis on $F$, we must have that a composition series for $V$ splits over $F[A(G)\cap N]$, and so \[\Res^G_{A(G)\cap N} V\cong \Res^G_{A(G)\cap N}S_1\oplus\cdots \oplus\Res^G_{A(G)\cap N} S_k.\]  Note that since $F$ is algebraically closed, each simple $F[A(G)\cap N]$-module is one-dimensional. Hence, by Clifford's theorem, we can find characters $\chi_1,\cdots,\chi_k\in \wh {A(G)\cap N}$, transversals $T_i$ to $G/I(\chi_i)$ and integers $e_i\geq 1$ so that $\Res^G_{A(G)\cap N} V$ affords a representation equivalent to  $\bigoplus_{i=1}^k\bigoplus_{t\in T_i}(t\chi_i)^{e_i}$.  As this representation is faithful, we deduce that \[\{1\}\leq \bigcap_{i=1}^k\bigcap_{g\in G}\ker g\chi_i\leq \bigcap_{i=1}^k\bigcap_{t\in T_i}\ker t\chi_i = \{1\}.\]   Thus (8) implies (6).

The most difficult implication is that (6) implies (7).  So let $\chi_1,\ldots,\chi_k\in \wh{A(G)\cap N}$ be such that $\bigcap_{i=1}^k\bigcap_{g\in G}\ker g\chi_i=\{1\}$.  Then by Proposition~\ref{p:build.faithful.rep.soc}, there exist simple $F[S(G)\cap N]$-modules $V_1,\ldots, V_k$ such that $\bigoplus_{i=1}^k\bigoplus_{t\in T_i}tV_i$ is $(S(G)\cap N)$-faithful, where $T_i$ is a transversal to $G/I(V_i)$ for $i=1,\ldots, k$.  Choose, for each $i=1,\ldots, k$, a simple quotient $W_i$ of $\Ind_{S(G)\cap N}^G V_i$.  We claim $W=W_1\oplus\cdots \oplus W_k$ affords a representation $\rho$ of $G$ with $\rho|_N$ faithful (it is clearly completely reducible with $k$ irreducible constituents).  If $\ker \rho\cap N\neq \{1\}$, then it contains a minimal normal subgroup of $G$, and hence  it suffices to check that $\ker\rho\cap S(G)\cap N=\{1\}$, (i.e., $\rho|_{S(G)\cap N}$ is faithful).  First note that by Frobenius reciprocity, $0\neq \Hom_{FG}(\Ind^G_{S(G)\cap N} V_i, W_i)\cong \Hom_{F[S(G)\cap N]}(V_i,\Res^G_{S(G)\cap N} W_i)$ and so $V_i$ is a submodule of $\Res^G_{S(G)\cap N} W_i$.  It follows from Clifford's theorem that $\Res_{S(G)\cap N}^G W_i\cong \bigoplus_{t\in T_i}(tV_i)^{e_i}$ for some integer $e_i\geq 1$.  Therefore, $\Res^G_{S(G)\cap N} W$ contains $\bigoplus_{i=1}^k\bigoplus_{t\in T_i}tV_i$ as a submodule and hence $\rho|_{S(G)\cap N}$ is faithful, as was required.  This completes the proof.
\end{proof}

We recall here some basic ingredients of the theory of Schur indices.  The reader can refer to~\cite[Theorem~74.5]{CurtisReinerII} for details.  Let $F$ now be an arbitrary field and $\ov F$ an algebraic closure of $F$.   Let $G$ be a finite group of order $n$ and let $\mathrm{Irr}(G)$ denote the set of irreducible characters of $G$ over $\ov F$. Denote by $\mu_n$ the group of $n^{th}$-roots of unity in $\ov F$.  Then $F(\mu_n)/F$ is a finite Galois (in fact, abelian) extension and each irreducible character in $\mathrm{Irr}(G)$ takes values in $F(\mu_n)$.  Let $H=\mathrm{Gal}(F(\mu_n)/F)$ be the Galois group.  Then $H$ acts on $\mathrm{Irr}(G)$ by $(\sigma\chi)(g)= \sigma(\chi(g))$ for $\chi$ a character and $\sigma \in H$.  Notice that the stabilizer of $\chi$ in $H$ has fixed field the character field $F(\chi)$, that is, the subfield generated over $F$ by $\chi(G)$.  We can extend $\sigma$ to an automorphism of $\ov F$, which we abusively denote by $\sigma$, as well; the identity map can of course be extended to the identity.  If  $\rho\colon G\to  GL_n(\ov F)$ is an irreducible representation, then we can define a new representation $\sigma\rho$ where $(\sigma\rho(g))_{ij}= \sigma(\rho(g)_{ij})$.   Clearly, $\sigma\rho$ is also irreducible, $\ker \rho=\ker\sigma\rho$ and if we replace $\rho$ by an equivalent representation $\rho'$, then $\sigma\rho$ is equivalent to $\sigma\rho'$.  One calls $\sigma\rho$ a \emph{Galois conjugate} of the representation $\rho$.   Note that the character of $\sigma\rho$ is $\sigma\chi$ where $\chi$ is the character of $\rho$.  If $V$ is the simple module corresponding to $\rho$, then $\sigma V$ denotes the simple module corresponding to $\sigma\rho$.  Note that if $\sigma,\tau\in H$, then $\sigma(\tau V)\cong (\sigma\tau V)$ and so $H$ acts on the isomorphism classes of simple $\ov FG$-modules.    Since two simple modules over an algebraically closed field are isomorphic if and only if they have the same character, this action is $H$-equivariantly isomorphic to the action of $H$ on $\mathrm{Irr}(G)$.  In particular, the stabilizer of the isomorphism class of $V$ is $\mathrm{Gal}(F(\mu_n)/F(\chi))$ where $\chi$ is the character of $V$, and hence the orbit of the isomorphism class of $V$ is in bijection with $\mathrm{Gal}(F(\chi)/F)$.

The theory of Schur indices states that if $V$ is a simple $\ov FG$-module with character $\chi$, then there is an integer $m(\chi)$, called the \emph{Schur index} of $\chi$, and a simple $FG$-module $W_{\chi}$  such that if $\sigma_1,\ldots, \sigma_r$ extend the elements of $\mathrm{Gal}(F(\chi)/F)$ to $F(\mu_n)$ with $\sigma_1$ the identity, then \[\ov F\otimes_F W_{\chi}=\left[\sigma_1 V\oplus\cdots \oplus \sigma_r V\right]^{m(\chi)}.\] Notice that $\sigma_1 V=V$ is a summand in $\ov F\otimes_F W_{\chi}$.    Moreover,  every simple $FG$-module is isomorphic to one of the form $W_{\chi}$. Of course, Galois conjugate characters give rise to the same simple $FG$-module.

With this setup, we can now drop the hypothesis that $F$ is algebraically closed from Theorem~\ref{t:alg.closed.case}.  This theorem was proved by \v{Z}mud$'$~\cite{Zmud1,zmud2} in the case that $G=N$ and the characteristic of $F$ does not divide the order of $G$.  The general result, in the case $G=N$, can be deduced from~\cite{zmud2} combined with~\cite{Nakayamafaithful}, but our proof is direct.

\begin{Thm}\label{t:gen.case}
Let $G$ be a finite group, $N\lhd G$ a normal subgroup and $F$ a field whose characteristic does not divide $|A(G)\cap N|$. Let $k\geq 1$.  Then the following are equivalent.
\begin{enumerate}
\item $S(G)\cap N$ is $k$-generated as a normal subgroup of $G$.
\item $A(G)\cap N$ is $k$-generated as a $\mathbb ZG$-module.
\item $A(G)_p\cap N$ is $k$-generated as an $\mathbb F_pG$-module for each prime divisor $p$ of $|A(G)|$.
\item $G$ has a completely reducible representation over $F$ with $k$ irreducible constituents that restricts to a faithful representation of $N$.
\end{enumerate}
\end{Thm}
\begin{proof}
In light of Theorem~\ref{t:alg.closed.case}, it suffices to show that if $\ov F$ is an algebraic closure of $F$, then $G$ has a completely reducible representation over $F$ with $k$ irreducible constituents that restricts to a faithful representation of $N$ if and only if it has one over $\ov F$.

Suppose first that $V=V_1\oplus\cdots\oplus V_k$  is an $N$-faithful $\ov FG$-module with $V_1,\ldots, V_k$ simple and  let $\chi_i$ be the character of $V_i$, for $i=1,\ldots, k$.  Let $W=W_{\chi_1}\oplus\cdots\oplus W_{\chi_k}$; it affords a completely reducible representation $\rho$ over $F$ with $k$ irreducible constituents.  Then $\rho$ is also the matrix representation afforded by \[\ov F\otimes_F W \cong (\ov F\otimes_F W_{\chi_1})\oplus\cdots\oplus (\ov F\otimes_F W_{\chi_k}).\]  Since $V_i$ is a summand in $\ov F\otimes_F W_{\chi_i}$, this latter representation has a submodule isomorphic to $V$ and hence $\rho|_N$ is faithful.

Conversely, suppose that $G$ has a completely reducible representation over $F$ with $k$ irreducible constituents that restricts to a faithful representation of $N$.  Then  we can find irreducible characters $\chi_1,\ldots, \chi_k$ of $G$ over $\ov F$ such that $W=W_{\chi_1}\oplus\cdots \oplus W_{\chi_k}$ affords a representation $\lambda$ of $G$ over $F$ with $\lambda|_N$ faithful.  Let $V_i$ be a simple $\ov FG$-module affording the character $\chi_i$ and let $\rho_i$ be a corresponding irreducible representation.  If $\sigma V_i$ is a Galois conjugate of $V_i$, then $\sigma\rho_i$ is the corresponding irreducible representation and it has the same kernel as $\rho_i$.  Thus $\ker \rho_i$ is the kernel of the representation afforded by \[\ov F\otimes_F W_{\chi_i}\cong \left[\sigma_1 V_i\oplus\cdots \oplus \sigma_r V_i\right]^{m(\chi_i)}\] where the $\sigma_jV_i$ run over the (isomorphism classes of) Galois conjugates of $V_i$.  As $\ov F\otimes_F W\cong (\ov F\otimes_F W_{\chi_1})\oplus\cdots\oplus (\ov F\otimes_F W_{\chi_k})$ affords the representation $\lambda$ of $G$ (viewed now as a representation over $\ov F$ instead of $F$), we deduce that the representation $\rho=\rho_1\oplus\cdots\oplus \rho_k$ afforded by $V_1\oplus\cdots \oplus V_k$
 has the same kernel as $\lambda$, and hence restricts to faithful representation of $N$, as required.
\end{proof}

We remark that the lemmas in Nakayama~\cite{Nakayamafaithful} (which are only applied there to the case $N=G$), can be used to give a characterization of when $G$ has a completely reducible representation with $k$ composition factors that restricts to a faithful representation of $N$ in the language used by Shoda~\cite{Shoda} and Tazawa~\cite{Tazawa}.

Specializing to the case $N=G$, we obtain the \v{Z}mud$'$'s theorem, suitably generalized.

\begin{Cor}[\v{Z}mud$'$]\label{c:zhmud}
Let $G$ be a finite group and $k\geq 1$.  Let $F$ be a field of characteristic $\ell$.  Then the following are equivalent.
\begin{enumerate}
\item $\ell\nmid |A(G)|$ and $S(G)$ is $k$-generated as a normal subgroup of $G$.
\item $\ell\nmid |A(G)|$ and $A(G)$ is $k$-generated as a $\mathbb ZG$-module.
\item $\ell\nmid |A(G)|$ and $A(G)_p$ is $k$-generated as an $\mathbb F_pG$-module for each prime divisor $p$ of $|A(G)|$.
\item $G$ has a faithful completely reducible representation over $F$ with $k$ irreducible constituents.
\end{enumerate}
\end{Cor}

In particular, specializing to the case $k=1$, we obtain Gasch\"{u}tz's theorem~\cite{Gaschutz} (but see~\cite{Szechtman} for a history of equivalent results proved by others, earlier).

\begin{Cor}[Gasch\"utz]\label{c:gaschutz}
Let $G$ be a finite group and $F$ a field of characteristic $\ell$.  Then the following are equivalent.
\begin{enumerate}
\item $\ell\nmid |A(G)|$ and $S(G)$ is generated as a normal subgroup of $G$ by a single element.
\item $\ell\nmid |A(G)|$ and $A(G)$ is a cyclic $\mathbb ZG$-module.
\item $\ell\nmid |A(G)|$ and $A(G)_p$ is a cyclic $\mathbb F_pG$-module for each prime divisor $p$ of $|A(G)|$.
\item $G$ has a faithful irreducible representation over $F$.
\end{enumerate}
\end{Cor}

\section{Faithful completely reducible representations of semigroups}\label{s:semigroup}
In this section, we extend \v{Z}mud$'$'s theorem to semigroups with a faithful completely reducible representation.  Rhodes characterized  the semigroups with a faithful irreducible representation over any field and with a faithful completely reducible representation over a field of characteristic zero in~\cite{Rhodeschar}.  Semigroups with a faithful completely reducible representation over an arbitrary field were characterized by Almeida, Margolis, Volkov and the author in~\cite{AMSV}.  All semigroups in this section shall be assumed finite.

\subsection{Green's relations and generalized group mapping congruences}
The preliminaries discussed here can be found in~\cite[Appendix~A]{qtheor} or~\cite[Chapter~1]{repbook}.
If $S$ is a semigroup, then $S^1$ denotes the monoid obtained by adjoining an identity to $S$  (even if it already had one).  Green's relations~\cite{Green} are fundamental to semigroup theory.  If $s,t\in S$, then $s\J t$ if $S^1sS^1=S^1tS^1$.   Note that $S/{\J}$ is partially ordered by $J\leq J'$ if $J\subseteq S^1J'S^1$.  A $\J$-class $J$ is called \emph{regular} if it contains an idempotent.  This is equivalent to each element of $J$ being von Neumann regular. We denote the set of regular $\J$-classes by $\mathrm{Reg}(S/{\J})$.   Fix an idempotent $e_J$ of each regular $\J$-class $J$ and let $G_J$ be the group of units of the monoid $e_JSe_J$ with identity $e_J$.  Then $G_J$ is called a \emph{maximal subgroup} of $J$ and it is independent of the choice of $e_J$ up to isomorphism.

Green's relations $\eL$ and $\R$ are defined by  $s\eL t$ if $S^1s=S^1t$ and $s\R t$ if $sS^1=tS^1$. If both $s\R t$ and $s\eL t$, then one writes $s\HH t$.  A key fact is that the $\HH$-class of $e_J$ is precisely $G_J$. We write $J_s$, $R_s$, $L_s$ and $H_s$ for the $\J$-, $\R$-, $\eL$- and $\HH$-classes of $s$, respectively.

 Finite semigroups enjoy the following property, termed \emph{stability}:
\[sS^1\cap J_s = R_s\qquad \text{and}\qquad S^1s\cap J_s=L_s.\]   A consequence of stability is that $s\J t$ if and only if there exists $r\in S$ with $s\R r\eL t$, if and only if there exists $r'\in S$ with $s\eL r'\R t$.
We shall also need to make use of Green's lemma~\cite{Green}.

\begin{Lemma}\label{l:Green}
Let $s\J t$.  Suppose that $s\eL r\R t$ with $r\in S$ and choose $u,w\in S^1$ with $r=us$ and $t=rw$.  Then $\psi\colon H_s\to H_t$ given by $\psi(x) = uxw$ is a bijection with $\psi(s) = t$.
\end{Lemma}

Krohn and Rhodes introduced generalized group mapping congruences in their work on Krohn-Rhodes complexity of finite semigroups~\cite{KRannals}.  It was Rhodes who recognized the relevance for representation theory~\cite{Rhodeschar}.

If $J$ is a regular $\J$-class of $S$, then the \emph{generalized group mapping congruence} $\equiv_J$ associated to $J$ is defined by $s\equiv_J t$ if, for all $x,y\in J$, \[xsy\in J\iff xty\in J\] and if $xsy,sty\in J$, then $xsy=xty$.  It is known that $\equiv_J$ is a congruence.  The reader is referred to~\cite[Section~4.6]{qtheor} for details about the congruence $\equiv_J$ and its universal property.  We denote by $\GGM$ the intersection of the congruences $\equiv_J$ over all regular $\J$-classes of $S$.

\subsection{Semigroup representation theory}
The reader is referred to~\cite[Chapter~5]{repbook} for the basic facts about representations of finite semigroups that we shall use.  This reference works with monoids, but the results for semigroups are the same, cf.~\cite[Chapter~5]{CP}. Let $K$ be a field and denote by $KS$ the semigroup  algebra of $S$ over $K$.  Note that $KS$ is a finite dimensional $K$-algebra but need not be unital.  Readers who are worried about such things may assume that $S$ is a monoid and they will not really miss out on anything.  If $R$ is a unital ring, then a  \emph{unitary} $R$-module $M$ is one where $1m=m$ for all $m\in M$. By a $KS$-module, we mean a $K$-vector space with an action of $S$ by $K$-linear maps (which then extends to $KS$ linearly); this is equivalent to a unitary $KS^1$-module.

Recall that a $KS$-module $V$ is \emph{simple} if $KS\cdot V\neq 0$ and the only submodules of $V$ are itself and the zero submodule.  A matrix representation afforded by a simple module is called \emph{irreducible}.  A \emph{semisimple} module is one isomorphic to a direct sum of simple modules. Note that the zero module is semisimple but not simple.  The simple summands are called \emph{irreducible constituents} or composition factors.   A  matrix representation afforded by a semsimple module is called \emph{completely reducible}.  The zero module affords a degree $0$ matrix representation.  A matrix representation of a semigroup is \emph{faithful} if it is an injective homomorphism.  Notice that the degree $0$ matrix representation is a faithful representation of both the trivial and the empty semigroup.

  The Jacobson radical of a ring $R$ will be denoted by $J(R)$.  The algebra $KS/J(KS)$ is semisimple and hence unital, and every semisimple $KS$-module is inflated from a $KS/J(KS)$-module.  The congruence associated to the composition $S\hookrightarrow KS\to KS/J(KS)$ is called the \emph{Rhodes radical} of $S$ with respect to $K$.  This congruence was described for the field of complex numbers in~\cite{Rhodeschar}, and for general fields in~\cite{AMSV}.  See also~\cite[Chapter~11]{repbook}.  Of course, $S$ admits a faithful completely reducible representation if and only if the Rhodes radical is trivial, that is, the equality relation.    However, we will avoid making specific reference to the results of~\cite{AMSV}, except for in a result showing that if the Rhodes radical is trivial, then the minimum length of a faithful representation can be achieved by a faithful completely reducible representation.

If $V$ is a simple $KS$-module, then a regular $\J$-class $J$ of $S$ is an \emph{apex} of $V$ if $sV=0$ if and only if $J\nsubseteq S^1sS^1$, i.e., $J\nleq J_s$. In other words, $J$ is the unique minimum $\J$-class not annihilating $V$.  Note that if $W$ is a $KS$-module and $J$ is a regular $\J$-class, then $e_JW$ is (a unitary) $KG_J$-module.   Clifford-Munn-Ponizovksy theory states the following.

\begin{Thm}[Clifford-Munn-Ponizovsky]\label{t:CMP}
Let $S$ be a finite semigroup and $K$ a field.
\begin{enumerate}
\item Every simple $KS$-module has an apex.
\item If $J$ is the apex of a simple $KS$-module $V$, then $e_JV$ is a simple $KG_J$-module.
\item Given a simple $KG_J$-module $V$, there is up to isomorphism one and only one simple $KS$-module $W$ with apex $J$ and $e_JW\cong V$.
\end{enumerate}
\end{Thm}

If $V$ is a simple $KG_J$-module, we shall usually denote by $\til V$ a simple $KS$-module with apex $J$ and $e_J\til V\cong V$.  The simple module $\til V$ is uniquely determined up to isomorphism.

The following result was first observed by Rhodes~\cite{Rhodeschar}.  We include a proof for completeness because it is fundamental to all that we do, and also because of its elegance. If $\alpha\colon S\to T$ is a semigroup homomorphism, then $\ker\alpha$, the \emph{kernel} of $\alpha$, is the associated congruence $s\mathrel{\ker \alpha} s'$ if $\alpha(s)=\alpha(s')$.

\begin{Lemma}[Rhodes]\label{l:rhodes}
Let $\rho\colon S\to M_n(K)$ be an irreducible representation with apex $J$.  Then ${\equiv_J}\subseteq {\ker \rho}$.
\end{Lemma}
\begin{proof}
Let $A$ be the unital subalgebra of $M_n(K)$ spanned by $\rho(S)$ and the identity matrix.  Then, since $\rho$ is irreducible, $K^n$ is a faithful simple $A$-module.  Therefore, $J(A)=0$ and so $A$ is semisimple. But a semisimple algebra with a faithful simple module is simple by the Wedderburn structure theorem.  By definition of an apex, $\rho(s)=0$ if and only if $J\nsubseteq S^1sS^1$.  Thus $\rho(S^1JS^1)\subseteq \rho(J)\cup \{0\}$ and hence $\rho(J)$ spans a nonzero ideal of the simple algebra $A$.  Therefore, the identity matrix is a linear combination $I=\sum_{j\in J} c_j\rho(j)$ of elements of $\rho(J)$.  Suppose that $s\equiv_J t$.  Then after multiplying on the left and right by our expression for $I$, we have that
\begin{align*}
\rho(s) &= \sum_{j,j'\in J}c_jc_{j'}\rho(jsj')\\
\rho(t) &= \sum_{j,j'\in J}c_jc_{j'}\rho(jtj').
\end{align*}
We claim that $\rho(jsj')=\rho(jtj')$ for all $j,j'\in J$.  First suppose that $\rho(jsj')=0$.  Then $J\nsubseteq S^1jsj'S^1$ and so $jsj'\notin J$.
Therefore, $jtj'\notin J$ by definition of $\equiv_J$, and so $\rho(jtj')=0$ as $J_{jtj'}< J$.  On the other hand, if $\rho(jsj')\neq 0$, then $J\leq J_{jsj'}\leq J$, and so $jsj'\in J$.  Therefore, $jsj'=jtj'$ by definition of $\equiv_J$ and so $\rho(jsj')=\rho(jtj')$.  It now follows that $\rho(s)=\rho(t)$, as required.
\end{proof}

As a consequence, we obtain the following result of Rhodes~\cite{Rhodeschar} (see also~\cite[Theorem~3.16]{AMSV}).

\begin{Cor}\label{c:faithful.ggm}
If a finite semigroup $S$ admits a faithful completely reducible representation over some field, then the congruence $\GGM$ is trivial.
\end{Cor}
\begin{proof}
It follows immediately from Theorem~\ref{t:CMP} and Lemma~\ref{l:rhodes} that $\GGM$ is contained in the kernel of any irreducible representation of $S$ and hence in the kernel of any completely reducible representation. The result follows. 
\end{proof}

By a \emph{null $KS$-module} we mean a one-dimensional module $W$ such that $KS\cdot W=0$.  Null modules are not simple but they can appear as composition factors in a composition series.    A \emph{composition series} \[0=V_0\leq V_1\leq \cdots \leq V_k=V\] for a $KS$-module $V$ is an unrefinable series of submodules.  Some of the composition factors $V_i/V_{i-1}$ can be null.  In fact, each $KS$-module is a unitary $KS^1$-module and a $KS$-composition series is the same thing as a $KS^1$-composition series in the usual sense.  The null composition factors correspond to the $KS^1$-composition factors with apex the $\J$-class of $1$, i.e., that are annihilated by $S$.
One last result that we shall need before turning to our main theorem is the following.

\begin{Prop}\label{p:proj.to.ss}
Let $S$ be a semigroup with a faithful completely reducible representation over a field $K$.  Let $V$ be a finite dimensional $KS$-module affording a faithful representation of $S$ and suppose that $V_1,\ldots, V_k$ are the non-null composition factors of $V$. Then $V_1\oplus\cdots\oplus V_k$ affords a faithful representation of $S$.
\end{Prop}
\begin{proof}
Let $0=W_0\leq W_1\leq\cdots \leq W_r=V$ be a composition series for $V$.
Let $\rho\colon S\to M_n(K)$ be a representation afforded by $V$ and let $A$ be the subalgebra of $M_n(K)$ spanned by $\rho(S)$.  Note that $S$ embeds via $\rho$ as a subsemigroup of $A$.  Moreover, each $W_i$ is naturally an $A$-module and so the $V_i$ are simple $A$-modules and the $KS$-module structure factors through $A$. The matrix representation of $KS$ afforded by $V_1\oplus\cdots\oplus V_k$ then factors through a homomorphism $\psi\colon A\to M_m(K)$.  We claim that $\ker \psi$ is nilpotent: in fact,  $\ker \psi^r=0$.  For if $a_1,\ldots, a_r\in \ker \psi$, then $a_i[W_i/W_{i-1}] =\{0\}$ since either $W_i/W_{i-1}$ is null, or $W_i/W_{i-1}\cong V_j$ for some $1\leq j\leq k$, and so $a_iW_i\subseteq W_{i-1}$.  Thus $a_1\cdots a_rV=a_1\cdots a_rW_r\subseteq W_0=0$.  It now follows from~\cite[Lemma~3.1]{AMSV},~\cite[Theorem~3.6]{AMSV} and triviality of the Rhodes radical that $\psi|_S$ is injective and so $V_1\oplus\cdots\oplus V_k$ affords a faithful representation of $S$.
\end{proof}

\subsection{A \v{Z}mud$'$ theorem for semigroups}
We now compute the minimum number of irreducible constituents of a faithful completely reducible representation of a finite semigroup, provided that it has one.  There are two motivating examples that are worth considering, as they led the author to the main result.  Suppose that $S$ is a finite meet semilattice, that is, $S$ is a finite poset with binary meets (greatest lower bounds).  We can make $S$ a semigroup via its meet operation.  Semilattices are precisely the commutative semigroups in which each element is idempotent.  A famous result of Solomon~\cite{Burnsidealgebra} shows that $KS\cong K^S$ and so $KS$ is semisimple and each simple $KS$-module is one-dimensional.  Thus the minimum number of irreducible constituents in a faithful representation of $S$ is precisely the minimum degree of a faithful representation of $S$.  This was computed by Mazorchuk and the author in~\cite{effective}.  An element $s\in S$ is \emph{join irreducible} if it cannot be expressed as a join (least upper bound) of some subset of elements strictly below it. Note that the minimum element of $S$ is not join irreducible since it is the join of the empty subset.   We showed~\cite{effective}  that the minimum degree of a faithful representation of $S$ is the number of join irreducible elements.  Note that each $\J$-class of a semilattice consists of a single idempotent, and $S$ and $S/{\J}$ are isomorphic posets.  Thus we should expect the partial order on regular $\J$-classes to play a role.

On the other hand, let $G$ be a nontrivial finite group and $N$ a proper, nontrivial normal subgroup.  The power set $P(G)$ is a semigroup under the usual multiplication of subsets: $AB=\{ab\mid a\in A,b\in B\}$.  Let $S$ be the subsemigroup consisting of the singletons and the elements of $G/N$ (viewed as cosets).  Identifying $g\in G$ with $\{g\}$, we can view $S$ as $G\cup G/N$ where elements of $G$ and $G/N$ multiply as usual and $g(hN) = ghN$, $(hN)g = hgN$ for $h,g\in G$. In particular, $G/N$ is an ideal of $S$.  Note that $S$ is a monoid with group of units $G$,  and the $\J$-classes of $S$ are $G$ and $G/N$, both of which are regular.  To simplify the discussion let us assume for the moment that we are working over a field of characteristic $0$.

Each irreducible representation of $G$ extends to $S$ by mapping $G/N$ to $0$, and these are the irreducible representations with apex $G$.  Also, each irreducible representation $\psi$ of $G/N$ extends to $S$ by sending $g\in G$ to $\psi(gN)$, and these are the irreducible representations with apex $G/N$.  Since all the irreducible representations with apex $G$ agree on $G/N$, in order to separate the elements of $G/N$ from each other we must use a direct sum of $k$ irreducible representations with apex $G/N$ where $k$ is the minimum number of normal generators of $S(G/N)$ by Corollary~\ref{c:zhmud}.    This direct sum will also separate elements of $G$ in different cosets of $N$.  Therefore, to separate the remaining elements of $G$ from each other and from $G/N$, we just need to find a completely reducible representation of $G$ whose restriction to $N$ is faithful with the minimum number of irreducible constituents (and extend it to $S$ by mapping $G/N$ to $0$). This number is given by the minimum number of normal generators of $S(G)\cap N$ by Theorem~\ref{t:gen.case}.  This was our motivation for proving Theorem~\ref{t:gen.case} in the first case.

We follow the convention here that an empty meet (join) in a lattice is the top (bottom) of the lattice.  In particular, we interpret the intersection of an empty collection of subsets of a set $X$ to be the set $X$.

Let us define a regular $\J$-class $J$ to be \emph{reducible} if \[\bigcap_{J'\in \mathrm{Reg}(S/{\J}), J'<J} {\equiv_{J'}}\subseteq {\equiv_J}.\]   Equivalently, $J$ is reducible if whenever $s\not\equiv_J t$, there is a regular $\J$-class $J'<J$ with $s\not\equiv_{J'} t$.  If a regular $\J$-class  is not reducible, then we say it is \emph{irreducible}.  The following proposition motivates the terminology.

\begin{Prop}\label{p:join.irr}
Let $S$ be a meet semilattice and $e\in S$.  Then $J_e$ is irreducible if and only if $e$ is join irreducible.
\end{Prop}
\begin{proof}
If $e\in S$, we claim that $s\equiv_{J_e} t$ if and only if either both $s,t\geq e$, or both $s,t\ngeq e$.  For this, note that $exe\in J_e$, if and only if $exe=e$, which occurs if and only if $x\geq e$.  The claim then follows by definition of $\equiv_{J_e}$  Suppose first that $e$ is not join irreducible.  Then $e$ is the join of the set $X$ of elements that are strictly below it.  In particular, if $s\equiv_{J_f} t$ for all $f\in X$, then either $s,t\ngeq f$ for some $f\in X$, in which case $s,t\ngeq e$, or $s,t\geq f$ for all $f\in X$, in which case $s,t\geq e$.  It follows that $s\equiv_{J_e} t$ and hence $J_e$ is reducible, as $J_f<J_e$ if and only if $f<e$.

Conversely, assume that $J_e$ is reducible and let $X=\{f\mid f<e\}$.  We claim that $e$ is the join of $X$.  It suffices to show that if $e'\geq f$ for all $f\in X$, then $e'\geq e$.  Indeed, if $e'\geq f$ for all $f\in X$, then $e'\equiv_{J_f} e$ for all $f< e$, i.e., for all regular $\J$-classes $J_f<J_e$.  Hence, by reducibility of $J_e$, we have that $e'\equiv_{J_e} e$ and thus, since $e\geq e$, we must have $e'\geq e$.  This completes the proof.
\end{proof}

The definition of an irreducible $\J$-class takes care of the order considerations.  Now we must handle the group theoretic ones.  Let $J$ be an irreducible regular $\J$-class with maximal subgroup $G_J$ and corresponding identity element $e_J\in J$.  Put \[N_J=\{g\in G_J\mid g\equiv_{J'} e_J, \forall J'\in \mathrm{Reg}(S/{\J})\ \text{with}\ J'<J\}\] and observe that it is a normal subgroup of $G_J$ since each $\equiv_{J'}$ is a congruence. Note that $N_J$ does not depend on the choice of $e_J$ in the following sense.  If $e\in J$ is an idempotent, then by standard finite semigroup theory, there exist $a,a'\in J$ with $aa'a=a$, $a'aa'=a'$, $aa'=e$ and $a'a=e_J$.  Moreover, $\psi\colon G_J\to G_e$ given by $\psi(x) = axa'$ is a group isomorphism with inverse $x\mapsto a'xa$.  It then follows that $g\in N_J$ if and only if $\psi(g)\equiv_{J'} e$ for all regular $\J$-classes $J'<J$.

We shall need two lemmas that will allow us to describe which completely reducible representations are faithful in order to establish our main result.

\begin{Lemma}\label{l:cantsep}
Let $J$ be a regular $\J$-class of $S$.  Suppose that $s,t\in S$ with $s,t\in S^1JS^1$ and $s\equiv_{J'} t$ for all regular $\J$-classes $J'<J$.  Then $\rho(s)=\rho(t)$ for any irreducible representation of $S$ with apex not equal to $J$.
\end{Lemma}
\begin{proof}
Let $J'$ be the apex of $\rho$.  If $J'<J$, then $\rho(s)=\rho(t)$ by Lemma~\ref{l:rhodes}.  On the other hand, if $J'\nleq J$, then $\rho(s)=0=\rho(t)$ by definition of an apex.
\end{proof}

Our second lemma establishes the faithfulness criterion.

\begin{Lemma}\label{l:is.faithful}
Let $K$ be a field and $S$ a finite semigroup.  Let $V$ be a semisimple $KS$-module and, for each regular $\J$-class $J$, let $V_J$ be the direct sum of the irreducible constituents of $V$ with apex $J$.  Note that $e_JV_J$ is a semisimple $KG_J$-module.  Then $V$ affords a faithful representation of $S$ if and only if:
\begin{enumerate}
\item $\GGM$ is trivial;
\item $V_J\neq 0$ whenever $J$ is irreducible;
\item $e_JV_J$ is  $N_J$-faithful whenever $J$ is irreducible and $N_J\neq \{e_J\}$.
\end{enumerate}
\end{Lemma}
\begin{proof}
The fact that $e_JV_J$ is a semisimple $KG_J$-module follows from Theorem~\ref{t:CMP}.
Suppose first that $V$ affords a faithful representation of $S$.  Then $\GGM$ is trivial by Corollary~\ref{c:faithful.ggm}.  Assume that $V=V_1\oplus \cdots \oplus V_k$  with $V_1,\ldots, V_k$ simple.  Let $J_i$ be the apex of $V_i$ and let $\rho_i$ be an irreducible representation afforded by $V_i$.  Let $J$ be an irreducible regular $\J$-class. Then we can find $s,t\in S$ with $s\not\equiv_J t$ but $s\equiv_{J'} t$ for all regular $\J$-classes $J'<J$.  By definition of $\equiv_J$, this means (possibly interchanging $s,t$), that we can find $x,y\in J$ with $xsy\in J$ and $xty\notin J$, or $xsy,sty\in J$ but $xsy\neq xty$.  In either case, $xsy\neq xty$. Note that $xsy\equiv_{J'} xty$ for all regular $\J$-classes $J'<J$.  It follows from Lemma~\ref{l:cantsep}  that $\rho_i(xsy)=\rho_i(xty)$ if $J_i\neq J$. We conclude, since $V$ affords a faithful representation and $xsy\neq xty$,  that some $J_i=J$.  This establishes (2). For (3), suppose that $N_J\neq \{e_J\}$.  Note that $e_JV$ is a $G_J$-faithful $KG_J$-module since $V$ affords a faithful representation of $S$, and if $g\in G_J$ acts trivially on $e_JV$, then $gv=ge_Jv=e_Jv$ for all $v\in V$, whence $g=e_J$ by faithfulness.     Note that $e_JV_J =\bigoplus_{J_i=J} e_JV_i$.  We claim that $e_JV_J$ is $N_J$-faithful. Indeed, if $g\in N_J$, then $g\equiv_{J'} e_J$ for all regular $\J$-classes $J'<J$, and so $\rho_i(g)=\rho_i(e_J)$ whenever $J_i\neq J$ by Lemma~\ref{l:cantsep}, whence $g$ acts trivially on $e_JV_i$ if $J_i\neq J$.  Thus $e_JV_J$ must be $N_J$-faithful as $e_JV=e_JV_1\oplus\cdots\oplus e_JV_k=e_JV_J\oplus \bigoplus_{J_i\neq J} e_JV_i$ affords a faithful representation of $G_J$, and hence $N_J$, but $N_J$ acts trivially on $\bigoplus_{J_i\neq J} e_JV_i$.

Conversely, suppose that (1)--(3) hold, and let $\rho$ be a representation afforded by $V$.   Let $s\neq t$ be elements of $S$.  Then since $\GGM$ is the trivial congruence, we can find a regular $\J$-class $J$ with $s\not\equiv_J t$.  We may assume without loss of generality that $J$ is minimal with respect to this property, that is, $s\equiv_{J'} t$ for all regular $\J$-classes $J'<J$.  In particular, it follows that $J$ is irreducible.  By definition of $\equiv_J$, we can find (possibly interchanging $s,t$) $x,y\in J$ with $xsy\in J$ and $xty\notin J$ or with $xsy,xty\in J$ and $xsy\neq xty$.  In the first case, $xsyV_J\neq 0$ and $xtyV_J=0$ since $0\neq V_J$ is a direct sum of simple modules with apex $J$ and $J_{xty}<J$.  Thus $\rho(xsy)\neq \rho(xty)$ and so $\rho(s)\neq \rho(t)$.  Next suppose that $xsy,xty\in J$ and $xsy\neq xty$.  Then $xsy,xty\in R_x\cap L_y$ by stability of finite semigroups and hence are $\HH$-equivalent.  By Green's lemma, we can find $u,v\in S^1$ with $r\mapsto urv$ a bijection from the $\HH$-class of $xty$ to $G_J$ with $uxtyv=e_J$.  Note that $g=uxsyv\in N_J$ as $g\equiv_{J'} e_J$ for all regular $\J$-classes $J'<J$, since $\equiv_{J'}$ is a congruence, and $g\neq e_J$.  Since $e_JV_J$ is $N_J$-faithful, we deduce that $g$ acts nontrivially on $e_JV_J$ and hence $\rho(g)\neq \rho(e_J)$.  It follows that $\rho(s)\neq \rho(t)$.  This concludes the proof that $\rho$ is faithful.
\end{proof}

We are now prepared to prove the main result of this paper.

\begin{Thm}\label{t:main}
Let $K$ be a field of characteristic $p\geq 0$ and let $S$ be a finite semigroup.
Then the following are equivalent for $k\geq 0$.
\begin{enumerate}
\item $S$ has a faithful completely reducible representation over $K$ with $k$ irreducible constituents.
\item $S$ has a faithful completely reducible representation over $K$ and there is a $KS$-module with $k$ simple composition factors affording a faithful representation of $S$.
\item $\GGM$ is trivial, $p\nmid |A(G_J)\cap N_J|$ for each irreducible regular $\J$-class $J$ and $k\geq \sum_{J\in \mathrm{Reg}(S/{\J})} k_J$ where: $k_J=0$ if $J$ is reducible; $k_J=1$ if $J$ is irreducible and $N_J=\{e_J\}$; and $k_J$ is the minimal number of generators of $S(G_J)\cap N_J$ as a normal subgroup of $G_J$, otherwise.
\end{enumerate}
\end{Thm}
\begin{proof}
The equivalence of the first two items follows from Proposition~\ref{p:proj.to.ss}.  Assume that (1) holds.  Then it is immediate from Lemma~\ref{l:is.faithful}, Proposition~\ref{p:faithful} and Theorem~\ref{t:gen.case} that $\GGM$ is trivial,  any faithful completely irreducible representation of $S$ has at least $k_J$ irreducible constituents with apex $J$ and  $p\nmid |A(G_J)\cap N_J|$ for any irreducible regular $\J$-class $J$.

Next we prove that (3) implies (1).  It clearly suffices to show that there is a faithful completely reducible representation with  $\sum_{J\in \mathrm{Reg}(S/{\J})} k_J$ irreducible constituents.  We proceed as follows.  If $J$ is an irreducible $\J$-class with $N_J=\{e_J\}$, then let $V_J$ be any simple module with apex $J$, for example, it could correspond as per Theorem~\ref{t:CMP} to the trivial representation of $G_J$.  If $J$ is irreducible and $N_J\neq \{e_J\}$, then choose $k_J$ simple $KG_J$-modules $V_1,\ldots, V_{k_J}$ such that $V_1\oplus\cdots \oplus V_{k_J}$ is $N_J$-faithful, using Theorem~\ref{t:gen.case}.  Put $V_J = \til V_1\oplus\cdots \oplus \til V_{k_J}$.  Then $V_J$ is a semisimple $KS$-module with $k_J$ irreducible constituents with apex $J$, and $e_JV_J\cong V_1\oplus\cdots \oplus V_{k_J}$ is an $N_J$-faithful $KG_J$-module.  It follows from Lemma~\ref{l:is.faithful} that $V=\bigoplus_J V_J$, where $J$ runs over the irreducible regular $\J$-classes, affords a faithful completely reducible representation $\rho$ of $S$
with the desired number of irreducible constituents.
\end{proof}

Notice that Theorem~\ref{t:main} recovers Corollary~\ref{c:zhmud} when $S$ is a group, and also recovers the result of~\cite{effective} when $S$ is a semilattice, in light of Proposition~\ref{p:join.irr}.

Theorem~\ref{t:main} implies that $S$ has a faithful completely reducible representation over $K$ if and only if $\GGM$ is trivial and, for each irreducible regular $\J$-class $J$, the characteristic of $K$ does not divide $|A(G_J)\cap N_J|$.  This may be easier to check in positive characteristic than directly computing the Rhodes radical as described in~\cite{AMSV}.

Let us give an example to illustrate Theorem~\ref{t:main}.  Let $G$ be a finite group and let $Q(G)=\bigcup_{N\lhd G}G/N$, which is submonoid of $P(G)$.  The normal subgroups of $G$ form a lattice.  A normal subgroup $N\lhd G$ is \emph{meet irreducible} if it cannot be expressed as an intersection of a (possibly empty) set $X$ of normal subgroups of $G$ that strictly contain $N$.  Note that $N$ is meet irreducible if and only if there is a unique smallest normal subgroup $\ov N$ with $N< \ov N$.  Indeed, let $N$ be a normal subgroup and $\ov N$ denote the intersection of all normal subgroups properly containing $N$.  Then $\ov N=N$ if $N$ is not meet irreducible and otherwise $N<\ov N$ and $\ov N$ has the desired property.  In the latter case, $\ov N/N$ is the unique minimal normal subgroup of $G/N$.

\begin{Prop}
Let $G$ be a finite group and $K$ a field of characteristic $p\geq 0$.  Then $Q(G)=\bigcup_{N\lhd G}G/N$ has a faithful completely reducible representation over $K$ if and only if $\ov M/M$ is not a $p$-group for any meet irreducible normal subgroup $M$.  Moreover, in this case the minimum number of irreducible constituents in a faithful completely reducible representation is the number of meet irreducible normal subgroups of $G$.
\end{Prop}
\begin{proof}
It is straightforward to verify (and well known) that the idempotents of $Q(G)$ are the normal subgroups and that if $M\lhd G$, then $J_M = G/M$. We claim that $g_1M_1\equiv_{J_M} g_2M_2$ if and only if either both $M_1,M_2\nleq M$, or both $M_1,M_2\leq M$ and $g_1M=g_2M$.  Indeed, if both $M_1,M_2\nleq M$, then $gMg_1M_1hM\notin J_M$ and $gMg_2M_2hM\notin J_M$ for any $gM,hM\in J_M$.  Thus   $g_1M_1\equiv_{J_M} g_2M_2$.  If $M_1,M_2\leq M$, then $gMg_1M_1hM=gg_1hM$ and $gMg_2M_2hM = gg_2hM$, and so $g_1M_1\equiv_{J_M} g_2M_2$ if and only if $g_1M=g_2M$. Finally, if say $M_1\leq M$ and $M_2\nleq M$, then $Mg_1M_1M\in J_M$ and $Mg_2M_2M\notin J_M$ and so $g_1M_1\not\equiv_{J_M}g_2M_2$.

It now follows easily that $\GGM$ is trivial. Indeed, if  $g_1M_1,g_2M_2\in Q(G)$  with $M_1\nleq M_2$, then $g_1M_1\not\equiv_{J_{M_2}} g_2M_2$ by the claim.  On the other hand, if $gM\neq hM$, then  $gM\not\equiv_{J_M}hM$ by the claim. Thus $\GGM$ is the trivial congruence.

Next we claim that $J_M$ is irreducible if and only if $M$ is meet irreducible.  First note that $J_M\leq J_{M'}$ if and only if $M'\leq M$.  Suppose that $\ov M=M$ and  $g_1M_1\equiv_{J_{M'}} g_2M_2$ for all $M'> M$. Suppose first that, say, $M_1\nleq  M$. Then since $M=\ov M$, we must have $M_1\nleq M'$ for some $M'>M$.  But then since $g_1M_1\equiv_{J_{M'}} g_2M_2$, we must have by the claim that $M_2\nleq M'$. Since $M<M'$, we conclude that $M_2\nleq M$. Thus $g_1M_1\equiv_{J_M} g_2M_2$ by the claim.  On the other hand, if $M_1,M_2\leq M$, then for each $M'>M$, we have $g_1M'=g_2M'$ by the claim.  Therefore, $g_1\inv g_2\in \ov M=M$ and so $g_1M=g_2M$.  We conclude that $g_1M_1\equiv_{J_M} g_2M_2$ by the claim, and so $J_M$ is reducible.

Conversely, suppose that $J_M$ is reducible.  We claim that $\ov M=M$.  Let $g\in \ov M$.  Then for every $M'>M$, we have $g\in M'$ and so $gM'=M'$.  Thus $gM\equiv_{J_{M'}} M$ for all $M'>M$ by the claim.  Therefore, by reducibility, $gM\equiv_{J_M} M$, and so $gM=M$ by another application of the claim. We conclude that $\ov M=M$, and hence $M$ is not meet irreducible.

Note that $J_M$ is the group $G/M$ and so $G_{J_M} = G/M$.    We claim that if $M$ is meet irreducible, then $N_{J_M}=\ov M/M$.   Indeed, if $gM\equiv_{J_{M'}} M$ for all $M'>M$, then $gM'=M'$ for all $M'>M$ by the claim, and so $g\in \ov M$.  Conversely, if $g\in \ov M$, then $gM'=M'$ for all $M'>M$, and so $gM\equiv_{J_{M'}} M$ for all $M'>M$ by the claim.  Note that since $\ov M/M$ is the unique minimal normal subgroup of $G/M$, we have $S(G_{J_M})= S(G/M) = \ov M/M= N_{J_M}$ is clearly generated by a single element as a normal subgroup.  Also, we have  $p\mid |A(G/M)\cap \ov M/M|$ if and only if $\ov M/M$ is an elementary abelian $p$-group by Corollary~\ref{c:abel.min.norm}.    The result now follows from Theorem~\ref{t:main}.
\end{proof}

We remark that the normal subgroup lattice of a group is modular and a famous theorem of Dilworth asserts that in a modular lattice the number of meet irreducible elements equals the number of join irreducible elements.

We next observe that Theorem~\ref{t:main} recovers Rhodes's characterization~\cite{Rhodeschar} of semigroups with a faithful irreducible representation.  A semigroup $S$ is \emph{generalized group mapping} if there is a regular $\J$-class $J$ (called the \emph{distinguished $\J$-class}) such that $\equiv_J$ is the trivial congruence.  In this case, either $J$ must be the minimal ideal of $S$, or $S$ must have a zero and $J$ is the unique minimal non-zero $\J$-class.  Moreover, in the former case $S$ acts faithfully on the left and right of $J$ and in the latter case it acts faithfully on the left and right of $J\cup \{0\}$ (and the converse is true, as well); see~\cite[Chapter~4.6]{qtheor} for details. The following lemma connects generalized group mapping semigroups with the language of the current paper.

\begin{Lemma}\label{l:group.mapping.struct}
Let $S$ be a nontrivial finite semigroup and $J$ a regular $\J$-class of $S$.  Then the following are equivalent.
\begin{enumerate}
\item $\GGM$ is the trivial congruence and $J$ is the unique irreducible regular $\J$-class of $S$.
\item $S$ is generalized group mapping with distinguished $\J$-class $J$.
\end{enumerate}
Moreover, if these equivalent conditions hold, then $N_J=G_J$.
\end{Lemma}
\begin{proof}
Suppose that (1) holds.  We show that $\equiv_J$ is trivial. Indeed, if $s\neq t$ in $S$, then since $\GGM$ is trivial,  $s\not\equiv_{J'} t$ for some regular $\J$-class $J'$, which we may assume is minimal with respect to this property.  But then $J'$ is irreducible and so $J'=J$.  Thus $\equiv_J$ is the trivial congruence, establishing (2).

Suppose that (2) holds.   Then $\GGM$ is trivial by definition.  We claim that $J$ is the unique irreducible regular $\J$-class and $N_J=G_J$. First note that if $J_s,J_t\ngeq J$, then $s\equiv_J t$ and so $s=t$.  Therefore, $J$ is either the minimal ideal of $S$, or $S$ has a zero element and $J$ is the unique minimal nonzero $\J$-class.  If $J$ is the minimal ideal, then since $S$ is nontrivial and $\equiv_J$ is trivial, we have that $J$ is irreducible. On the other hand, if $S$ has a zero and $J$ is the minimum nonzero $\J$-class, then $\equiv_{J_0}$ (where $J_0$ is the $\J$-class of zero) is the universal relation and $\equiv_J$ is trivial, and so $J$ is irreducible and $J_0$ is not irreducible.  Moreover, for all other regular $\J$-classes $J'$, we have that $J<J'$ and ${\equiv_J}\subseteq {\equiv_{J'}}$ since, $\equiv_J$ is trivial.  Thus $J$ is the only irreducible regular $\J$-class. Moreover, $N_J=G_J$ since either there are no regular $\J$-classes below $J$, or there is just $J_0$ and $\equiv_{J_0}$ is universal.  This completes the proof.
\end{proof}

\begin{Cor}[Rhodes]
Let $S$ be a nonempty finite semigroup.  Then $S$ has a faithful irreducible representation over a field $K$ if and only if it is generalized group mapping with distinguished $\J$-class $J$ such that the maximal subgroup $G_J$ has a faithful irreducible representation over $K$, that is, $G_J$ has no nontrivial normal $p$-subgroup, where $p$ is the characteristic of $K$, and $S(G_J)$ is generated by a single element as a normal subgroup.
\end{Cor}
\begin{proof}
The trivial semigroup has a faithful irreducible representation and is generalized group mapping with the maximal subgroup of the distinguished $\J$-class the trivial group and so the result is true when $S$ is trivial.  Assume from now on that $S$ is nontrivial.

Suppose first that $S$ is generalized group mapping with distinguished $\J$-class $J$ and $G_J$ admits a faithful irreducible representation. Then by Lemma~\ref{l:group.mapping.struct}, $J$ is the unique irreducible regular $\J$-class and $N_J=G_J$.   Thus $S$ has a faithful irreducible representation by Theorem~\ref{t:main} and Ga\-sch\"utz's Theorem (Corollary~\ref{c:gaschutz}). 

Next suppose that $S$ has a faithful irreducible representation.  Then by Theorem~\ref{t:main}, $\GGM$ is trivial and  $S$ must have a unique irreducible regular $\J$-class $J$.   But then $S$ is generalized group mapping with distinguished $\J$-class $J$ by Lemma~\ref{l:group.mapping.struct} and $N_J=G_J$. 
Therefore, $p\nmid |A(G_J)|$ and $S(G_J)$ can be generated by a single element, again by Theorem~\ref{t:main}. We conclude that $G_J$ has a faithful irreducible representation by Corollary~\ref{c:gaschutz}, completing the proof.
\end{proof}

\def\malce{\mathbin{\hbox{$\bigcirc$\rlap{\kern-7.75pt\raise0,50pt\hbox{${\tt
  m}$}}}}}\def\cprime{$'$} \def\cprime{$'$} \def\cprime{$'$} \def\cprime{$'$}
  \def\cprime{$'$} \def\cprime{$'$} \def\cprime{$'$} \def\cprime{$'$}
  \def\cprime{$'$} \def\cprime{$'$}

%\bibliographystyle{abbrv}
%\bibliography{standard2}
\end{document}